%% file: main.tex
\RequirePackage{fix-cm}
\documentclass[nospthms]{svjour3}       
\smartqed  

\usepackage{graphicx}
\usepackage{amsmath,amssymb,amsthm,stmaryrd}
\usepackage{enumerate}
\usepackage{mathtools}
\usepackage[all,2cell]{xy}
\usepackage{hyperref}
\usepackage{booktabs,siunitx}
\usepackage{algpseudocode}

\usepackage{siunitx}

\usepackage{tikz}
\usepackage{tikz-cd}
\usetikzlibrary{patterns}
\usetikzlibrary{positioning}
\usetikzlibrary{shapes,decorations.markings,arrows.meta}
\usetikzlibrary{matrix,arrows}
\tikzset{hfit/.style={rounded rectangle, inner xsep=0pt},
           vfit/.style={rounded corners}}

\usetikzlibrary{patterns}

\usepackage{caption}
\usepackage{subcaption}

\usepackage{pgfplots}
 
\pgfplotsset{compat = newest}

\usepackage{bm}

\input{shortcuts}

\begin{document}

\title{Morse Theoretic Templates for High Dimensional Homology Computation}


%



\author{Shaun Harker 
		\and Konstantin Mischaikow \and Kelly Spendlove 
}


\institute{S. Harker \at
              Department of Mathematics, Rutgers University, Piscataway, NJ 08854, USA \\
              \email{sharker@math.rutgers.edu}             
           \and
           K. Mischaikow \at
              Department of Mathematics, Rutgers University, Piscataway, NJ 08854, USA\\
              \email{mischaik@math.rutgers.edu}
            \and
            K. Spendlove \at
            Mathematical Institute, University of Oxford, Oxford, Oxfordshire OX2 6GG, UK\\
            \email{spendlove@maths.ox.ac.uk}
}

\date{Received: date / Accepted: date}

\maketitle

\begin{abstract}

We introduce the notion of a \emph{template} for discrete Morse theory.  Templates provide a memory efficient approach to the computation of homological invariants (e.g., homology, persistent homology, Conley complexes) of cell complexes.  We demonstrate the effectiveness of templates in two settings: first, by computing the homology of certain cubical complexes which are homotopy equivalent to $\mathbb{S}^d$ for $1\leq d \leq 20$, and second,  by computing Conley complexes and connection matrices for a collection of examples arising from a Conley-Morse theory on spaces of braids diagrams.

\end{abstract}

\input{intro}

\input{prelims}

\input{principles}

\input{hypercube}

\input{globalTemplate}

\input{experiments}
\begin{acknowledgements} KM  was 
partially supported by NSF grants  1521771, 1622401, 1839294, 1841324, 1934924, by NIH-1R01GM126555-01 as part of the Joint DMS/NIGMS Initiative to Support Research at the Interface of the Biological and Mathematical Science, DARPA contract HR0011-16-2-0033, and a grant from the Simons Foundation. KS was partially supported by the NSF Graduate Research Fellowship Program under grant DGE-1842213 and by EPSRC grant EP/R018472/1.
\end{acknowledgements}

\bibliographystyle{spmpsci}      
\bibliography{ref}   

\end{document}

%% file: shortcuts.tex
\newtheorem{thm}{Theorem}[section]

\newtheorem{defn}[thm]{Definition}

\newtheorem{prop}[thm]{Proposition}

\theoremstyle{remark}
\newtheorem{rem}[thm]{Remark}

\algdef{SE}[DOWHILE]{Do}{doWhile}{\algorithmicdo}[1]{\algorithmicwhile\ #1}%

\newcommand{\conley}{\textrm{conley}}

\newcommand{\bw}{{\bf w}}
\newcommand{\bx}{{\bf x}}
\newcommand{\by}{{\bf y}}

\newcommand{\K}{{\mathbb{K}}}
\newcommand{\N}{{\mathbb{N}}}
\newcommand{\R}{{\mathbb{R}}}

\newcommand{\Z}{{\mathbb{Z}}}

\newcommand{\cC}{{\mathcal C}}

\newcommand{\cF}{{\mathcal F}}

\newcommand{\cH}{{\mathcal H}}

\newcommand{\cK}{{\mathcal K}}

\newcommand{\cM}{{\mathcal M}}

\newcommand{\cS}{{\mathcal S}}

\newcommand{\cX}{{\mathcal X}}

\newcommand{\sP}{{\mathsf P}}

\newcommand{\be}{{\mathbf e}}

\newcommand{\id}{\text{id}}

\DeclareMathOperator{\st}{star}

{\begin{snugshade}\begin{quote}}
{\hfill\end{quote}\end{snugshade}}

\definecolor{shadecolor}{rgb}{0.8,0.8,0.8}

\usepackage{bm}
\newcommand{\balpha}{\bm{\alpha}}

\newcommand{\bbeta}{\bm{\beta}}

\newcommand{\uu}{{\bf u}}
\newcommand{\vv}{{\bf v}}

\newcommand{\bt}{{\bf t}}
\newcommand{\cross}{{\rm cross}}

\newcommand{\sSC}{\mathsf{SC}}

\newcommand{\Mate}{\textsc{Mate}}
\newcommand{\GMate}{\textsc{GradedMate}}
\newcommand{\GMorse}{\textsc{GradedMorse}}
\newcommand{\CM}{\textsc{ConnectionMatrix}}
\newcommand{\HOM}{\textsc{Homology}}

\newcommand{\Conf}{\mathbf{D}}
\newcommand{\rel}{~{\rm rel}~}

%% file: intro.tex
\section{Introduction}

The use of algebraic topological methods, and in particular homology, to analyze data is a rapidly expanding field as it provides systematic methods with which to codify complex and high-dimensional data in terms of algebraic invariants.
The same is true for the study of nonlinear dynamical systems where computational advances allow for the global analysis of high dimensional systems.
However, a major stumbling block to computing homology in high dimensional settings is that in applications the typical starting point is a cell complex $\cX$ and the number of elements $N$ in the complex grows rapidly with dimension.
Thus, memory constraints imply that explicitly generating a high dimensional cell complex is typically infeasible.
Furthermore, the worst case analysis for the algebraic operations required to compute homology is $O(N^3)$ or worse. 

A popular technique is to use discrete Morse theory~\cite{forman,focm,mn} to preprocess the cellular data, which in practice has the effect of typically producing close to linear time computational speeds for computing homology.
The key concept is that of an \emph{acyclic partial matching} $w\colon \cX\to \cX$ (see Definition~\ref{defn:acyclicmatching}) from which one can derive another, typically much smaller chain complex, $C_\bullet(w)$ with homology isomorphic to $H_\bullet(\cX)$.  The expectation (which is typically realized) is that the computational cost of computing $H_\bullet(C_\bullet(w))$ is significantly less than that of $H_\bullet(\cX)$.

An acyclic partial matching is an example of the more general {\em partial matching}, which is an involution $w\colon \cX\to \cX$ obeying an {\em incidence trichotomy}; the set $A(w)=\{\xi\mid w(\xi)=\xi\}$ is the basis for the chain complex $C_\bullet(w)$ when $w$ is acyclic.  Hence one wants to minimize the cardinality of $A(w)$.  However, it is known that in general computing an acyclic partial matching which minimizes $A(w)$, i.e., one that is \emph{optimal}, is NP-hard~\cite{joswig2006computing}, and can even be NP-hard to approximate~\cite{bauer2019hardness}. In this paper, we instead leverage one of the strengths of discrete Morse theory: given a fixed cell complex $\cX$, there are typically many admissible (acyclic) partial matchings on $\cX$.   In the case that one has a sequence of partial matchings $(w_1,w_2,\ldots,w_n)$, we give an algorithm $\Mate$ (see Section~\ref{sec:pr}) which aggregates the sequence into a single matching $w$, for which $\# A(w) \leq \#A(w_1)$. In Section~\ref{sec:pr}, we prove the following theorem regarding $\Mate$.

\begin{thm}\label{thm:algmatching}
Let $\cX$ be a cell complex and $\bw = (w_1,w_2,\ldots,w_n)$ a sequence of partial matchings.  Then $w(\cdot):=\textsc{Mate}(\cdot,\bw)$ is a $\bw$-stable partial matching.  Moreover, the algorithm $\textsc{Mate}(\cdot,\bw)$ executes in $O(2^n)$ time.
\end{thm}

While $\Mate$ produces a partial matching from a sequence of (acyclic) partial matchings, it need not be acyclic.  To prove acyclicity, more control over the particular sequence $\bw$ is needed.  We return to this point below.

Now, not only do we wish to minimize the size of $A(w)$, we also note that if $A(w)$ can be derived by examining $\cX$ locally, then the above mentioned memory costs can be avoided because one does not need to store the entire cell complex $\cX$ in memory at once.  To achieve this we develop the notion of a \emph{template}, by which we mean a predetermined partial matching; in this paper, we use the set of templates $\{\cM_i\}$ which are defined on particular cell complexes which we call \emph{hypercube complexes}.  The $n$-dimensional hypercube complex ($n$-cube complex), denoted $\cH_n$, consists of binary strings of length $n$.  The structure of the cell complex is defined in the following manner: a binary string $\bx= (x_i)$ with $x_i\in \{0,1\}$ is a cell, with dimension given by $\sum_i x_i$.  A cell $\bx=(x_i)$ is a face of $\by=(y_i)$ if and only if $x_i\le y_i$ for all $i$.  See Section~\ref{sec:prelim:hypercube} for more details. Fig.~\ref{fig:3cube} is an example of the Hasse diagram of the face poset of the 3-cube complex, $\cH_3$.

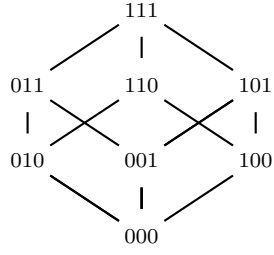
\begin{figure}[h!]
    \centering
    \begin{tikzpicture}
 \tikzstyle{vertex}=[circle,minimum size=20pt,inner sep=0pt]
 \tikzstyle{edge} = [draw,thick,-,black]
  \tikzstyle{medge} = [draw,dotted,black]
\def\x{1.5}
\def\y{1}
\node[vertex] (0) at (0,0) {$000$};

\node[vertex] (100) at (\x,\y) {$100$};
\node[vertex] (001) at (0,\y) {$001$};
\node[vertex] (010) at (-\x,\y) {$010$};

\node[vertex] (110) at (0,2*\y) {$110$};
 \node[vertex] (011) at (-\x,2*\y) {$011$};
 \node[vertex] (101) at (\x,2*\y) {$101$};

 \node[vertex] (111) at (0,3*\y) {$111$};
 
 \draw[edge] (0)--(010)--(011)--(111);
 \draw[edge] (0)--(010)--(110)--(111);
 \draw[edge] (0)--(001)--(011);
 \draw[edge] (0)--(001)--(101)--(111);
 \draw[edge] (0)--(001)--(101);
 \draw[edge] (0)--(100)--(101);
 \draw[edge] (100)--(110);
\end{tikzpicture}
 \caption{The Hasse diagram of the face poset of the $3$-cube complex $\cH_3$.  Cells are binary strings of length three. }   \label{fig:3cube}
\end{figure}


For an example of a template, consider the function $\cM_i\colon \cH_n\to \cH_n$ which flips the $i$th bit from the left, e.g., for $\cM_1\colon \cH_3\to \cH_3$ we have $\cM_1(000) = 100$ and $\cM_1(100) = 000$.  This is a predetermined acyclic partial matching, which is further formalized in Setion~\ref{sec:hypercube}.  We depict partial matchings visually as in Fig.~\ref{fig:ncube:matching}. 

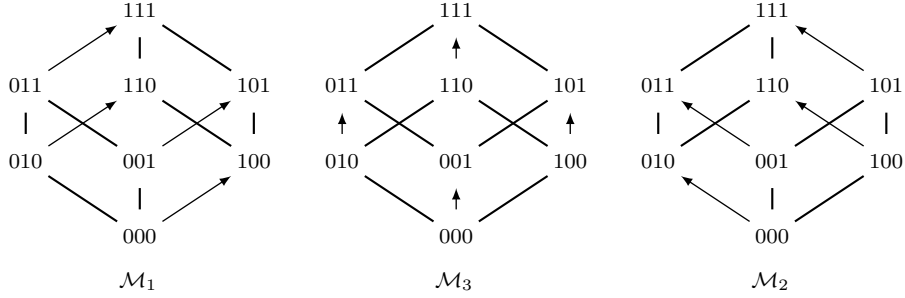
\begin{figure}[h!]
\begin{minipage}{.3\textwidth}
\centering
\begin{tikzpicture}
 \tikzstyle{vertex}=[circle,minimum size=20pt,inner sep=0pt]
 \tikzstyle{edge} = [draw,thick,-,black]
  \tikzstyle{medge} = [draw,dotted,black]
\def\x{1.5}
\def\y{1}
\node[vertex] (0) at (0,0) {$000$};

\node[vertex] (100) at (\x,\y) {$100$};
\node[vertex] (001) at (0,\y) {$001$};
\node[vertex] (010) at (-\x,\y) {$010$};

\node[vertex] (110) at (0,2*\y) {$110$};
 \node[vertex] (011) at (-\x,2*\y) {$011$};
 \node[vertex] (101) at (\x,2*\y) {$101$};

 \node[vertex] (111) at (0,3*\y) {$111$};
 \draw[edge] (0) -- (001);
 \draw[edge] (0)--(010);
 \draw[edge] (001)--(011);
  \draw[edge] (010)--(011);
 \draw[edge] (100)--(101);
  \draw[edge] (100)--(110);
 \draw[edge] (110)--(111);
  \draw[edge] (101)--(111);
 \draw[-latex,line width=.5pt] (0)--(100);
 \draw[-latex,,line width=.5pt] (001)--(101);
 \draw[-latex,,line width=.5pt] (011)--(111);
\draw[-latex,,line width=.5pt] (010)--(110);
\end{tikzpicture}

$\cM_1$
\end{minipage}
\hspace{3mm}
\begin{minipage}{.3\textwidth}
\centering
\begin{tikzpicture}
 \tikzstyle{vertex}=[circle,minimum size=20pt,inner sep=0pt]
 \tikzstyle{edge} = [draw,thick,-,black]
  \tikzstyle{medge} = [draw,dotted,black]
\def\x{1.5}
\def\y{1}
\node[vertex] (0) at (0,0) {$000$};

\node[vertex] (100) at (\x,\y) {$100$};
\node[vertex] (001) at (0,\y) {$001$};
\node[vertex] (010) at (-\x,\y) {$010$};

\node[vertex] (110) at (0,2*\y) {$110$};
 \node[vertex] (011) at (-\x,2*\y) {$011$};
 \node[vertex] (101) at (\x,2*\y) {$101$};

 \node[vertex] (111) at (0,3*\y) {$111$};
 \draw[edge] (0)--(010);
 \draw[edge] (001)--(011);
  \draw[edge] (100)--(110);
  \draw[edge] (101)--(111);
  \draw[edge] (0) -- (100);
  \draw[edge] (011) -- (111);
  \draw[edge] (010) -- (110);
  \draw[edge] (001)--(101);
 \draw[-latex,line width=.5pt] (0)--(001);
 \draw[-latex,,line width=.5pt] (100)--(101);
 \draw[-latex,,line width=.5pt] (110)--(111);
\draw[-latex,,line width=.5pt] (010)--(011);
\end{tikzpicture}

$\cM_3$
\end{minipage}
\hspace{3mm}
\begin{minipage}{.3\textwidth}
\centering
\begin{tikzpicture}
 \tikzstyle{vertex}=[circle,minimum size=20pt,inner sep=0pt]
 \tikzstyle{edge} = [draw,thick,-,black]
  \tikzstyle{medge} = [draw,dotted,black]
\def\x{1.5}
\def\y{1}
\node[vertex] (0) at (0,0) {$000$};

\node[vertex] (100) at (\x,\y) {$100$};
\node[vertex] (001) at (0,\y) {$001$};
\node[vertex] (010) at (-\x,\y) {$010$};

\node[vertex] (110) at (0,2*\y) {$110$};
 \node[vertex] (011) at (-\x,2*\y) {$011$};
 \node[vertex] (101) at (\x,2*\y) {$101$};

 \node[vertex] (111) at (0,3*\y) {$111$};
 \draw[edge] (0) -- (001);
  \draw[edge] (010)--(011);
 \draw[edge] (100)--(101);
 \draw[edge] (110)--(111);
  \draw[edge] (0)--(100);
  \draw[edge] (010)--(110);
  \draw[edge] (011)--(111);
  \draw[edge] (001)--(101);
 \draw[-latex,line width=.5pt] (0)--(010);
 \draw[-latex,,line width=.5pt] (001)--(011);
 \draw[-latex,,line width=.5pt] (100)--(110);
\draw[-latex,,line width=.5pt] (101)--(111);
\end{tikzpicture}

$\cM_2$
\end{minipage}
 \caption{The templates $\{\cM_i\}$ on the $3$-cube complex $\cH_3$.  The partial matching $\cM_i$ is visualized by a set of directed arrows, with $\bx\to \by$ if $\cM_i(\bx)=\by$ and $\bx\leq\by$.}   \label{fig:ncube:matching}
\end{figure}

In Section~\ref{sec:hypercube} we prove the following result, which shows that the templates on $\cH_n$ enable the construction of acyclic partial matchings on any cell complex $\cX$ which is embeddable into $\cH_n$ (see Definition~\ref{defn:embedding}).

\begin{thm}\label{thm:hypercube}
Let $\cX$ be a cell complex which is embeddable into $\cH_n$ with embedding $\iota\colon \cX\to \cH_n$.  For any $1\leq i \leq n$, the function $\alpha_i\colon \cX\to \cX$ defined as
\begin{equation}\label{eqn:hypercube:alpha}
\alpha_i(\xi) =
\begin{cases}
   \iota^{-1}(\cM_i(\iota (\xi))) &\text{ if $\cM_i(\iota(\xi))\in \iota(\cX)$},\\
    \xi      & \text{otherwise,}
\end{cases}
\end{equation}
is an acyclic partial matching.  Moreover, the function $w\colon \cX\to \cX$ given by $w(\cdot):=\Mate(\cdot,\balpha)$, where $\balpha=(\alpha_1,\alpha_2,\ldots,\alpha_n)$, is an acyclic partial matching.
\end{thm}

Observe that the use of the hypercube complexes allows us to guarantee the construction of an acyclic partial matching.

We now wish to generalize the use of templates from  $\cH_n$-embeddable cell complexes to an arbitrary cell complex $\cX$.  We show that this is possible if $\cX$ can be decomposed in such a way that each part of the decomposition is embeddable into a hypercube.  In Section~\ref{sec:general} we prove the following result.

\begin{thm}\label{thm:global}
Let $\cX$ be a cell complex and $f\colon (\cX,\leq)\to \sP$ an order-preserving map, where $(\cX,\leq)$ is the face poset of $\cX$ and $\sP$ is a poset.  If for each $p\in \sP$ there is an $n$ such that the fiber $\cX^p = f^{-1}(p)$ is embeddable into $\cH_n$,  then the function $w\colon \cX\to \cX$ defined by
\[
w(\xi ) = \Mate (\xi, \balpha_p) \text{ if and only if } \xi \in \cX^p,
\]
where the $\balpha_p$ are given by \eqref{eqn:hypercube:alpha} is an acyclic partial matching on $\cX$.
\end{thm}

In Section~\ref{sec:general} we also describe schemes by which to decompose both cubical and simplicial complexes in order to use Theorem~\ref{thm:global}.  

Finally, in Section~\ref{sec:experiments} we perform computational experiments using our implementation, available at~\cite{cmcode}, for two sets of examples:
\begin{enumerate}
    \item[(i)] Computing the homology of cubical complexes $\cS^d,\cS_{top}^d \subset \R^{d+1}$ which are homotopy equivalent to $d$-spheres $\mathbb{S}^d$ for $1\leq d \leq 20$. See Section~\ref{sec:experiments:hypersphere}.
    \item[(ii)] Computing Conley complexes (see Section~\ref{sec:braids}) for graded cubical complexes $\cC(m-1;d)\subset \R^d$ associated to braid diagrams $\vv_n,\bt^m$ arising from a Conley-Morse theory on braids diagrams~\cite{im,braids}.
\end{enumerate}

We note that these can be sizable computations, with one of the largest examples for (i) being $\cS^{20}\subset \R^{21}$ containing on the order of $3\times 10^9$ cells, and the largest example for (ii) being $\cC(5;10)\subset \R^{10}$, containing on the order of $2.5\times 10^{10}$ cells.  Such examples are possible due to the low memory overhead of the templates scheme.  Moreover in both set of examples we see that our templates-based algorithm has roughly linear performance with respect to $N$, the number of cells. See Fig.~\ref{plot:intro}.  

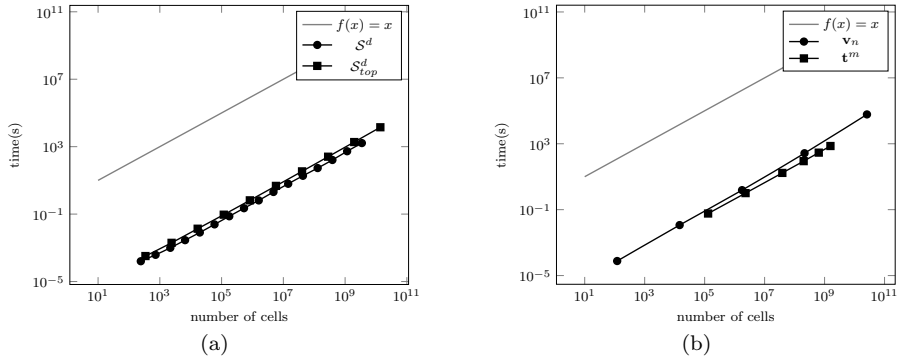
\begin{figure}
\begin{minipage}{.48\textwidth}
\centering
\begin{tikzpicture}[scale=.65]
\begin{axis}[
    xlabel= number of cells,
    ylabel= time(s),
    xmode=log,
    ymode=log,
]
\addplot[
    domain = 10:10000000000,
    samples = 200,
    smooth,
    thick,
    gray,
] {x};
\addplot[
    smooth,
    mark=*,
    thick,
    black,
] file[skip first] {data_hypercube.dat};
\addplot[
    smooth,
    mark=square*,
    thick,
    black,
] file[skip first] {data_top_hypercube.dat};

\legend{
    $f(x)=x$,
    $\cS^d$,
    $\cS_{top}^d$
}
\end{axis}
\end{tikzpicture}\\
(a)
\end{minipage}
\hspace{\fill}
\begin{minipage}{.48\textwidth}
\flushright
\centering
\begin{tikzpicture}[scale=.65]
\begin{axis}[
    xlabel= number of cells,
    ylabel= time(s),
    xmode=log,
    ymode=log,
]
\addplot[
    domain = 10:10000000000,
    samples = 200,
    smooth,
    thick,
    gray,
] {x};
\addplot[
    smooth,
    mark=*,
    thick,
    black,
] file[skip first] {data_nfold_braid.dat};
\addplot[
    smooth,
    mark=square*,
    thick,
    black,
] file[skip first] {data_torus_braid.dat};
\legend{
    $f(x)=x$,
    $\vv_n$,
    $\bt^m$
}
\end{axis}
\end{tikzpicture}\\
(b)
\end{minipage}
\caption{Time(s) vs. number for cells for (a) $ \cS^d,\cS_{top}^d\subset \R^{d+1}$,  and (b) cubical complex $\cC(m-1;d)\subset \R^d$ associated to braid diagrams $\vv_n,\bt^m$ (see Section~\ref{sec:braids}). Log-log scale.}\label{plot:intro}
\end{figure}

%% file: prelims.tex
\section{Preliminaries}

\subsection{Cell Complexes}\label{sec:prelims:cell}

We make use of the following cell complex, cf.~\cite[Chapter III (Definition 1.1)]{lefschetz}, which we define over a field $\K$. 

\begin{defn}
\label{defn:cellComplex}
{\em
A {\em cell complex} $(\cX,\leq,\kappa,\dim)$ is a finite poset  $(\cX,\leq)$ together with two associated functions $\dim\colon \cX\to \N$ and $\kappa\colon \cX\times \cX\to \K$ subject to the following conditions:
\begin{enumerate}
\item $\dim\colon(\cX,\leq)\to (\N,\leq)$ is an order preserving map, i.e., $\xi\leq \xi'$ implies $\dim \xi \leq \dim \xi'$,
\item  For each $\xi$ and $\xi'$ in $\cX$:
\[
\kappa(\xi,\xi')\neq 0\quad\text{implies } \xi'\leq \xi \quad\text{and}\quad \dim(\xi) = \dim(\xi')+1,
\]
\item\label{cond:three} For each $\xi$ and $\xi''$ in $\cX$,
\[
\sum_{\xi'\in \cX} \kappa(\xi,\xi')\cdot \kappa(\xi',\xi'')=0.
\]
\end{enumerate}

For simplicity we typically write $\cX$ for $(\cX,\leq,\kappa,\dim)$.  
The partial order $\leq$ is the {\em face partial order}.   An element $\xi\in \cX$ is called a {\em cell} and $\dim \xi$ is the {\em dimension} of $\xi$. The set $\cX$ is graded with respect to $\dim$, i.e., $\cX = \bigsqcup_{n\in \N} \cX_n$ with $\cX_n = \dim^{-1}(n)$.   If $\xi\leq \xi'$ then $\xi'$ is a \emph{face} of $\xi$.  If in addition  $\dim (\xi') = \dim (\xi) - 1$ then $\xi' $ is a {\em primary face} of $\xi$.  The function $\kappa$ is the {\em incidence function} of the complex. The values of $\kappa$ are referred to as the {\em incidence numbers}.    
}
\end{defn}

\begin{rem}
Note that Theorems~\ref{thm:algmatching}--\ref{thm:global} hold when the field $\K$ is replaced by a principal ideal domain $\mathbf{R}$ and the appropriate change is made to the definitions of incidence relation and partial matching.\footnote{Namely, $\xi'\prec \xi$ if and only if $\kappa(\xi,\xi')=u$ for a unit $u\in \mathbf{R}$.}  However, working over fields both simplifies the exposition and allows for the use of iterative discrete Morse theory to compute both homology and Conley complexes, as in the algorithms $\HOM$ and $\CM$ of Section~\ref{sec:experiments}.
\end{rem}

The incidence function $\kappa$ engenders a binary relation, which we call the {\em incidence relation}.

\begin{defn}\label{defn:prelim:ir}
{\em
Let $\cX$ be a cell complex. 
The \emph{incidence relation} is the binary relation on $\cX$ denoted by $\prec$ and given as 
\[
\xi' \prec \xi \text{ if and only if } \kappa(\xi,\xi')\neq 0 .
\]
}
\end{defn}

We introduce the concept of an \emph{incidence embedding}, which is a very weak notion of an injective map of cell complexes; in particular an incidence embedding need not induce a chain map.

\begin{defn}\label{defn:embedding}
{\em
Let $\cX$ and $\cX'$ be cell complexes. An {\em incidence embedding} $\iota\colon \cX\to \cX'$ is an injective map such that for all $\xi\in \cX$,
\[
\xi' \prec \xi \text{ if and only if } \iota(\xi')\prec \iota(\xi) .
\]
If there is an incidence embedding $\iota\colon \cX\to \cX'$ we say that $\cX$ is {\em embeddable} into $\cX'$, or simply that $\cX$ is \emph{$\cX'$-embeddable}.
}
\end{defn}

Let  $\cX\subset \cX'$ where $\cX'$ is a cell complex.  If $(\cX,\leq,\kappa,\dim)$, meaning the restrictions to $\cX$, is in itself a cell complex, then $\cX$ is called a \emph{subcomplex} of $\cX'$. In this case, the inclusion map $\cX\hookrightarrow \cX'$ is an incidence embedding, but does not necessarily  induce a chain map.

Given a cell complex $\cX$ the {\em associated chain complex $C_\bullet(\cX)$} is the chain complex $C_\bullet(\cX) = \{C_n(\cX)\}_{n\in\Z}$ where $C_n(\cX)$ is the vector space over $\K$ with basis elements given by the cells $\xi\in \cX_n$, i.e., 
\[
C_n(\cX) = \bigoplus_{\xi\in \cX_n} \K\langle \xi\rangle,
\]
and the boundary operator $\partial_n\colon C_n(\cX) \to C_{n-1}(\cX)$ is defined by
\[
\partial_n( \xi) := \sum_{\xi' \in \cX} \kappa(\xi, \xi')\xi'.
\]
Condition~(\ref{cond:three}) of Definition~\ref{defn:cellComplex} ensures $\partial_{n-1}\circ \partial_n = 0$.   Given a cell complex $\cX$ the {\em homology} of $\cX$, denoted $H_\bullet(\cX)$, is defined as the homology of the associated chain complex $H_\bullet(C_\bullet(\cX))$.

\subsection{Hypercube Complex}\label{sec:prelim:hypercube}

In this section we describe the \emph{hypercube complex}, which is a very useful cell complex.   Fix $n\geq 1$. Consider the group $\Z_2=\{0,1\}$ with operation addition modulo two and partial order $(\Z_2,\leq)$ where $\leq$ is the partial order generated by $0\leq 1$.  Denote by $\Z_2^n=\{0,1\}^n$ the direct product with component-wise addition modulo two. For a vector $\bx\in \Z_2^n$ we write $\bx =(x_i)$ with $x_i\in \{0,1\}$. The product partial order $\leq$ on $\Z_2^n$ is given by
\begin{equation}\label{eqn:product:order}
(x_1',\ldots,x_n')\leq (x_1,\ldots,x_n)  \text{ if and only if } x_i' \leq x_i \text{ for all } i.
\end{equation}
Let $\be_i=(e_j)$ denote the vector in $\Z_2^n$ such that $e_i = 1$ and $e_j=0$ for all $j\neq i$. 

\begin{defn}
{\em
The \emph{$n$-dimensional hypercube complex} (or $n$-cube complex), denoted $(\cH_n,\leq,\kappa,\dim)$, is the cell complex with underlying set of cells $\cH_n=\Z_2^n$. The partial order $\leq$ is the product order given in \eqref{eqn:product:order}. The incidence function $\kappa\colon \cH_n\times \cH_n\to \Z_2$ is given by 
\[
\kappa(\bx,\bx') = 
\begin{cases}
1 & \text{if } \bx' \text{ is a primary face of } \bx,\\
0 & \text{otherwise.}
\end{cases}
\]
The function $\dim\colon \cH_n\to \N$ is given by taking $\bx=(x_i)$ to 
\[
\dim (\bx) = \sum_{1\leq i \leq n} x_i .
\]

}
\end{defn}
\begin{rem}
Note that our $n$-cube complex is defined with the incidence numbers in the field $\Z_2$. This suffices for the purposes of the paper. \end{rem}





\subsection{Cubical Complexes}\label{prelims:cubes}

Cubical complexes often arise in computational dynamics as grids on a phase space, e.g.~\cite{kmm}, and therefore have a particular importance  for Conley theory computations which we perform in Section~\ref{sec:braids}. In this section we give a brief review of cubical complexes which follows~\cite{kmm}.  An \emph{elementary interval} is a subset $I\subset \R$ of the form $I=[l,l+1]$ or $I=[l,l]$ for some $l\in \Z$.  An {\em elementary cube} $\xi$ in $\R^n$ is a finite product of elementary intervals, i.e., 
\[
\xi = I_1\times I_2\times \cdots \times I_n \subseteq \R^n .
\]
A {\em cubical set} is a union of elementary cubes.   Intervals of length zero are called {\em degenerate} while those of length 1 are {\em nondegenerate}. The {\em dimension} of a cube $\xi$, denoted $\dim (\xi)$, is the number of nondegenerate intervals in $\xi$. This furnishes a function $\dim\colon \cX\to \N$. The face poset is defined by $\xi' \leq \xi$ if $\xi'\subseteq \xi$.   Given a particular field $\K$, an incidence function $\kappa\colon \cX\times \cX\to \K$ can be described in detail~\cite{kmm}.   When $\K=\Z_2$ the incidence number $\kappa(\xi,\xi')$ is nonzero if and only if $\xi'$ is a primary face of $\xi$.

\begin{defn}
 {\em
 A {\em cubical complex} is a cubical set such that $(\cX,\leq,\kappa,\dim)$ form a cell complex.
 }
\end{defn}


\subsection{Discrete Morse Theory}\label{sec:dmt}

We review the use of discrete Morse theory to compute homology of complexes. Our exposition is brief and follows~\cite{focm,mn}. 

\begin{defn}\label{defn:acyclicmatching}
{\em
A {\em partial matching} of cell complex $\cX$ is an involution $w\colon \cX\to \cX$ that is subject to the following \emph{incidence trichotomy}:
\begin{enumerate}
  \item $w(\xi)=\xi$, or
     \item $\xi \prec w(\xi)$, or
     \item $ w(\xi) \prec \xi$.
\end{enumerate}
}
\end{defn} 

A partial matching $w$ engenders three sets, 
\[
A(w) =\{\xi \mid \xi = w(\xi)\}, \quad\quad 
Q(w) =\{\xi\mid \xi\prec w(\xi)\},  \quad\quad K(w) = \{\xi \mid w(\xi) \prec \xi   \}.
\]

The sets $A(w),Q(w),K(w)$ partition $\cX$ and the restriction $w|_{Q(w)}\colon Q\to K$ is a bijection (hence the name partial matching), cf.~\cite{focm}.  There is a binary relation $\ll$ on $Q(w)$ defined by 
\[
\xi_1\ll \xi_0 \text{ if and only if } \xi_1 \prec w(\xi_0) .
\]

A partial matching $w$ is called \emph{acyclic} if the transitive closure $\ll^+$ of $\ll$ generates a partial order on $Q$. By convention, partial matchings are visualized as in Fig.~\ref{fig:ncube:matching}, with directed arrows $\xi\to \xi'$ if $\xi\in Q$, $\xi'\in K$ and $w(\xi)=\xi'$.

\begin{defn}
{\em 
Let $w$ be a partial matching.  A \emph{discrete flowline}, or simply \emph{flowline}, is a sequence of distinct cells
\[
\xi_0 \prec w(\xi_0) \succ \xi_1\prec w(\xi_1)\prec \ldots \succ \xi_n\prec w(\xi_n).
\]
}
\end{defn}

Note that if such a flowline exists if and only if  $\xi_n\ll^+ \xi_0$.

 \begin{rem}
 Partial matchings are sometimes called {\em discrete vector fields} and acyclic partial matchings are sometimes called {\em discrete gradient vector fields}.  In this language, a cell $\xi\in A(w)$ is called a {\em critical cell} and is a fixed point which respect to the flowlines.
 \end{rem}
 
 The upshot of discrete Morse theory, is that there is a chain complex $C_\bullet(w)$ which is defined using the critical cells and flowlines and is chain homotopy equivalent to $C_\bullet(\cX)$, see~\cite[Theorem 3.10]{focm}. In particular, $H_\bullet(C_\bullet(w))\cong H_\bullet(\cX)$.  The chain complex $C_\bullet(w)$ is called the \emph{Morse (chain) complex}.

\subsection{Graded Morse Theory}\label{sec:dmt:graded}

In this section, we review a graded version of discrete Morse theory.  


\begin{defn}
{\em
Let $\sP$ be a poset and $f\colon (\cX,\leq)\to \sP$ an order-preserving map.  A partial matching is {\em $\sP$-graded}, or simply {\em graded}, if it satisfies the property that $w(\xi)=\xi'$ only if $\xi,\xi'\in \cX^p =f^{-1}(p)$ for some $p\in \sP$.    That is, matchings may only occur in the same fiber of the grading.  
}
\end{defn}

A $\sP$-graded acyclic partial matching produces a $\sP$-filtered chain equivalence (that is, a chain equivalence which respects the grading) between $C_\bullet(w)$ and $C_\bullet(\cX)$, when these are regarded as $\sP$-graded chain complexes, see~\cite{hms}, cf.~\cite{mn}.  In~\cite{hms} graded acyclic partial matchings are used to compute Conley complexes, and in~\cite{mn} they are used as a preprocessing step before the computation of persistent homology.

Similar ideas to graded matchings can be found many places in the literature, see~\cite{koz}. The basic idea is that one can assemble an acyclic partial matching on the complex $\cX$ from a collection of acyclic partial matchings on the fibers of an order-preserving map.  We restate~\cite[Patchwork Theorem]{koz} in our language.

\begin{thm}[Patchwork Theorem]\label{thm:dmt:patchwork}

Let $\sP$ be a poset and $f\colon (\cX,\leq)\to \sP$ an order-preserving map.  If $\{w^p\colon \cX^p\to \cX^p\}_{p\in \sP}$ is a collection of acyclic partial matchings the subcomplexes $\cX^p=f^{-1}(p)$, then $w\colon \cX\to \cX$, defined as
\[
w(\xi) = w^p(\xi) \text{ if and only if } \xi \in \cX^p,
\]
is a $\sP$-graded acyclic partial matching on $\cX$.
\end{thm}

%% file: principles.tex
\section{The Mate Algorithm}\label{sec:pr}

In this section we describe a recursive algorithm -- \textsc{Mate} -- which takes as input a sequence of partial matchings $\bw = (w_1,w_2,\ldots,w_n)$ and derives a new partial matching which aggregates the matching information of $\bw$.  This leads to the proof of  Theorem~\ref{thm:algmatching}.   For the remainder of this section, let $\cX$ be a cell complex and $\bw = (w_1,w_2,\ldots,w_n)$ be a sequence of partial matchings on $\cX$.

\begin{defn}
{\em
A partial matching $w$ on $\cX$ is {\em $\bw$-compatible} if for each $\xi\in \cX$ there is an $i$ such that 
$w(\xi)=w_i(\xi)$.
}
\end{defn}

\begin{rem}
Note that each $w_i$ is an $\bw$-compatible matching, and therefore set of $\bw$-compatible matchings is nonempty.   
\end{rem}

Consider a $\bw$-compatible partial matching $w$.
Since $w$ is a partial matching there is a relation $\gg$ on $Q(w)$. In particular if $\xi_0\gg \xi_1$, then there exists a flowline
\[
\xi_0 \prec w(\xi_0)\succ \xi_1 \prec w(\xi_1).
\]
As $w$ is $\bw$-compatible there exists $i$ and $j$ such that
\[ 
w_i(\xi_0)=w(\xi_0) \text{ and } w_j(\xi_1) = w(\xi_1).
\]
We incorporate this information into the flowline via the following labeling,
\[
\xi_0 \stackrel{i}{\prec} w(\xi_0)\succ \xi_1 \stackrel{j}{\prec} w(\xi_1) .
\]

A notion stronger than compatibility is that of {\em stability} -- a matching which obeys the ordering of $\bw$. Stability forces the flowlines to obey a particular forcing condition, which becomes useful when proving acyclicity, see Theorem~\ref{thm:hypercube}.


\begin{defn}\label{defn:stable}
{\em
Let $w$ be an $\bw$-compatible matching.  Consider $\xi_0\gg \xi_1$ with associated flowline
\[
\xi_0 \stackrel{j}{\prec} w(\xi_0)\succ \xi_1 \stackrel{j'}{\prec} w(\xi_1) .
\]
The pair $\xi_0\gg \xi_1$ is {\em $\bw$-unstable} if there exists $i$ such that
\begin{enumerate}
    \item $w(\xi_0)=w_i(\xi_1)$, and
    \item $i<j$ and $i<j'$.
\end{enumerate}
A matching is called {\em $\bw$-stable} if there are no $\bw$-unstable pairs.
}
\end{defn}


\begin{prop}\label{prop:stablematching}

For any sequence $\bw = (w_1,w_2, \ldots,w_n)$ of partial matchings on $\cX$ there is a $\bw$-stable matching $w$ on $\cX$.

\end{prop}
\begin{proof}
We inductively derive a sequence of partial matchings $(w_{\le 0},w_{\le 1}\ldots,w_{\le n})$.   First, set $w_{\leq 0}:= \id $. Let $A_{i} = A(w_{\leq i})$ and 
\begin{equation}\label{eqn:matching}
w_{\leq i}(\xi) = 
\begin{cases}
      w_i(\xi) & \text{if $\xi,w_i(\xi)\in A_{i-1}$}\\
      w_{\leq  i-1}(\xi) & \text{ otherwise.} 
\end{cases} 
\end{equation}
The matching of interest is 
\begin{equation}\label{eqn:final}
w:= w_{\leq n}.
\end{equation}

We now show that $w$ is a stable partial matching, beginning by proving that $w$ is a partial matching.
We proceed by induction on the sequence $(w_{\leq 0},\ldots,w_{\leq n})$.  

The base case is $w_{\le 0} = \id$, and thus $w_{\le 0}$ is a partial matching with $A_0 = \cX$.  Now we assume that $w_{\leq i-1}$ is a partial matching and we show that $w_{\le i}$ is also a partial matching.
This amounts to showing two things:  first, that $w_{\le i}$ is an involution, and second, that it obeys the incidence trichotomy.    
Note that by the inductive hypothesis we have that $w_{\leq i-1}^2=\id$.  If $\xi,w_i(\xi)\in A_{i-1}$ then
\[
w_{\leq i}(w_{\leq i}(\xi)) = w_{\leq i}(w_i(\xi)) = w_i^2(\xi) = \xi .
\]
Otherwise $w_{\le i}^2(\xi)=w_{\le i-1}^2(\xi)=\id$.  
Therefore $w$ is an involution.   
Now, if $\xi,w_i(\xi)\in A_{i-1}$ then $w_{\le i}(\xi)=w_i(\xi)$.  
As $w_i$ is a partial matching by hypothesis, the incidence trichotomy holds. 
Otherwise, $w_{\le i}(\xi)=w_{\le i-1}(\xi)$, and as $w_{\le i-1}$ is a partial matching by the inductive hypothesis, the incidence trichotomy holds. 
Therefore $w_{\le i}$ is a partial matching. 

It remains to show that $w$ is $\bw$-stable.  First, $w$ is $\bw$-compatible from \eqref{eqn:matching}.  Now suppose that $w$ had an unstable pair $\xi_0\gg\xi_1$ with associated flowline
\[
\xi_0 \stackrel{j}{\prec} w(\xi_0)\stackrel{i}{\succ} \xi_1 \stackrel{j'}{\prec} w(\xi_1),
\]
where $w_i(\xi_1)=w(\xi_0)$ and $i<j,j'$.  Let $\xi' = w(\xi_0)$ and $\xi'' = w(\xi_1)$.  As $w_i(\xi_1)=\xi'$ but $w(\xi_1)\neq \xi'$ it follows from~\eqref{eqn:matching} that either $\xi_1\not\in A_{i-1}$ or $\xi'\not\in A_{i-1}$.  That $\xi'=w(\xi_0)=w_j(\xi_0)$ implies $\xi_0,\xi'\in A_{j-1}$.  Similarly, $\xi'' = w(\xi_1)=w_{j'}(\xi_1)$ implies $\xi_1,\xi''\in A_{j'-1}$.  By~\eqref{eqn:matching} there is a descending sequence
\[
\cX=A_0\supseteq A_1 \supseteq\ldots \supseteq A_n.
\]
By assumption $i<j,j'$, which implies $A_j,A_{j-1}\subseteq A_{i-1}$.  Therefore $\xi_1,\xi'\in A_{i-1}$, which is a contradiction. Therefore $w$ is $\bw$-stable.
\end{proof}

\begin{rem}
The terminology of unstable/stable comes from the similarity of Definition~\ref{defn:stable} to the stable marriage problem~\cite{stbmrg:knuth}.  In $\Mate$ we take a greedy approach, however this similarity may help in designing better algorithms.  We leave this to future work.
\end{rem}


The construction of the stable matching $w$ as defined in Eqn.~\eqref{eqn:final} can be expressed as the recursive algorithm, $\textsc{Mate}(\cdot,\bw)$. Note that $\textsc{MateHelper}$ is a helper function which performs the recursion and streamlines the notation for $\Mate$. 

    \begin{algorithmic}
 \label{alg:mate}
    \Function{Mate}{$\xi,\bw=(w_1,w_2,\ldots,w_n$)}
        \State \Return \textsc{MateHelper}$(\xi,n,\bw)$
    \EndFunction
    \Function{MateHelper}{$\xi$, $i$, $\bw$}
        \If{$i=0$}
            \State \Return $\xi$
        \EndIf
        \State $\xi' \gets \textsc{MateHelper}(\xi,i-1,\bw)$
        \If {$\xi'=\xi$}
            \If {$\textsc{MateHelper}(w_i(\xi),i-1,\bw)=w_i(\xi)$ }
                \State \Return $w_i(\xi)$
            \EndIf
        \EndIf
        \State \Return $\xi'$
    \EndFunction
    \end{algorithmic}



\begin{proof}[Proof of Theorem~\ref{thm:algmatching}]
Notice $\textsc{MateHelper}(\cdot,i,\bw)=w_{\le i}(\cdot)$.  In particular, $w_{\leq 0} = \id = \textsc{MateHelper}(\cdot,0,\bw)$ and $w(\cdot)=\Mate(\cdot,\bw)$ is the matching in~\eqref{eqn:final}.  It follows from the proof of Proposition~\ref{prop:stablematching} that $w$ is a $\bw$-stable partial matching.

We prove the time complexity bound by induction on $n$. In the worst case, the two {\bf if} statements are true at each level of recursion.   We define a function $T$ for the complexity of the algorithm in the worst case. For $n=0$, $\xi$ is returned. Set $T(0) = c_1$ and $T(n) = c_2+2T(n-1)$.  The closed form of this recurrence relation is $T(n) = 2^nc_1 + c_2\sum_{i=0}^{n-1} 2^i = 2^n(c_1+c_2)-c_2 = O(2^n)$.
\end{proof}

\subsection{Algorithm for Graded Complexes}

In the graded setting, the input is a cell complex $\cX$, a poset $\sP$, and an order-preserving map $f\colon (\cX,\leq)\to \sP$. We may adapt the algorithm \textsc{Mate} to the graded case, which we call $\GMate$, by using a sequence of $\sP$-graded partial matchings, recall from Section~\ref{sec:dmt:graded} that a partial matching is {\em $\sP$-graded} if two cells are matched only if they belong to the same fiber.

\begin{prop}\label{prop:gradedmate}
Let $\cX$ be a cell complex, $\sP$ be a poset and $f\colon (\cX,\leq)\to \sP$ be an order-preserving map.  If $\bw=(w_1,w_2,\ldots,w_n)$ is a sequence of $\sP$-graded partial matchings then $w(\cdot)=\Mate(\cdot,\bw)$ is a $\sP$-graded partial matching.

\end{prop}
\begin{proof}
If follows from the proof of Theorem~\ref{thm:algmatching} that $w$ is $\bw$-compatible, thus for $\xi\in \cX^p$ there is an $i$ such that $w(\xi)=w_i(\xi)$.  As each $w_i$ is $\sP$-graded, it follows that $w$ is $\sP$-graded.  
\end{proof}





%% file: hypercube.tex
\section{Templates for the Hypercube Complex }\label{sec:hypercube}

Recall from the introduction that the templates for the hypercube complexes are the functions $\cM_i\colon \cH_n\to \cH_n$ defined as
 \[
\cM_i(\bx) = \bx+\be_i,
 \]
where $+$ denotes the group operation in $\Z_2^n$, see Fig.~\ref{fig:ncube:matching}.  It is straightforward that $\cM_i$ is an involution, obeys the incidence trichotomy and is acyclic.  Finally, note that $A(\cM_i) = \varnothing$.  For implementation, it is often useful to consider the vectors of $\Z_2^n$ as bit strings; in that case the function $\cM_i$ is a NOT operation on the $i$th bit from the left.  

A crucial observation for this paper is that incidence embeddings can be used to pull back partial matchings. This allows for the templates to be used to create partial matchings for any $\cH_n$-embeddable complex $\cX$ via Theorem~\ref{thm:hypercube}, see Fig.~\ref{fig:subncubematch}.

\begin{prop}\label{prop:incidenceembed}
 Let $\cX$ and $\cX'$ be cell complexes and let $\iota\colon \cX\to \cX'$ be an incidence embedding.  If $w'\colon \cX'\to \cX'$ is a partial matching, then 
 \begin{equation}\label{eqn:incidenceembed}
w(\xi) =
\begin{cases}
   \iota^{-1}(w'(\iota (\xi))) &\text{ if $w'(\iota(\xi))\in \iota(\cX)$},\\
    \xi      & \text{otherwise,}
\end{cases}
\end{equation}
is a partial matching on $\cX$.  If $w'$ is acyclic then $w$ is acyclic.
 \end{prop}
 \begin{proof}
 The map $w$ is an involution since if $w'(\iota(\xi))\in \iota(\cX)$ then
 \[
w\circ w (\xi) = 
 \iota^{-1}\circ w' \circ \iota \circ \iota^{-1} \circ w' \circ \iota (\xi) = 
  \iota^{-1}\circ w'  \circ w' \circ \iota (\xi) =
  \iota^{-1}  \circ \iota(\xi) =
   \xi.
 \]
 Otherwise, $w(w(\xi))=w(\xi)=\xi$.  
 Moreover, as $\iota$ is an incidence embedding, it follows that $w$ obeys the incidence trichotomy.     Now assume $w'$ is acyclic.  
 If $w$ is not acyclic then for some $\xi_0\in \cX$, we have $\xi_0 \gg^+ \xi_0$, i.e., there is a sequence $\xi_0\gg \xi_1 \gg\ldots \gg \xi_n \gg \xi_0$ in $\cX$.  This corresponds to a flowline
 \[
 \xi_0 \prec w(\xi_0) \succ \xi_1\prec w(\xi_1)\succ \ldots \succ \xi_n \prec w(\xi_n)\succ \xi_0.
 \]
 As $\iota$ is an incidence embedding, this flowline can be pushed forward to a flowline for $w'$ in $\cX'$:
  \[
 \iota (\xi_0) \prec \iota(w(\xi_0)) \succ \iota(\xi_1) \prec \iota( w(\xi_1)) \succ \ldots \succ \iota(\xi_n) \prec \iota(w(\xi_n))\succ \iota(\xi_0).
 \]
 However, this is a contradiction as $w'$ is acyclic.  Therefore $w$ is acyclic.
 \end{proof}
 
\begin{figure}[h!]
\begin{minipage}{.48\textwidth}
    \centering
    \begin{tikzpicture}
 \tikzstyle{vertex}=[circle,minimum size=20pt,inner sep=0pt]
 \tikzstyle{edge} = [draw,thick,-,black]
  \tikzstyle{medge} = [draw,dotted,black]
\def\x{1.5}
\def\y{1}
\node[vertex] (0) at (0,0) {$000$};

\node[vertex] (100) at (\x,\y) {$100$};
\node[vertex] (001) at (0,\y) {$001$};
\node[vertex] (010) at (-\x,\y) {$010$};

\node[vertex] (110) at (0,2*\y) {};
 \node[vertex] (011) at (-\x,2*\y) {$011$};
 \node[vertex] (101) at (\x,2*\y) {$101$};

 \node[vertex] (111) at (0,3*\y) {};
 
 \draw[edge,dotted] (011)--(111);
 \draw[edge,dotted] (101)--(111);
 \draw[edge,dotted] (110)--(111);
 \draw[edge] (0)--(010)--(011);
 \draw[edge, dotted] (0)--(010)--(110);
 \draw[edge] (0)--(001)--(011);
 \draw[edge] (0)--(001)--(101);
 \draw[edge] (0)--(001)--(101);
 \draw[edge] (0)--(100)--(101);
 \draw[edge, dotted] (100)--(110);
\end{tikzpicture}

(a)
\end{minipage}
\begin{minipage}{.48\textwidth}
    \centering
    \begin{tikzpicture}
 \tikzstyle{vertex}=[circle,minimum size=20pt,inner sep=0pt]
 \tikzstyle{edge} = [draw,thick,-,black]
  \tikzstyle{medge} = [draw,dotted,black]
\def\x{1.5}
\def\y{1}
\node[vertex] (0) at (0,0) {$000$};

\node[vertex] (100) at (\x,\y) {$100$};
\node[vertex] (001) at (0,\y) {$001$};
\node[vertex] (010) at (-\x,\y) {$010$};

\node[vertex] (110) at (0,2*\y) {};
 \node[vertex] (011) at (-\x,2*\y) {$011$};
 \node[vertex] (101) at (\x,2*\y) {$101$};

 \node[vertex] (111) at (0,3*\y) {};
 
 \draw[edge,dotted] (011)--(111);
 \draw[edge,dotted] (101)--(111);
 \draw[edge,dotted] (110)--(111);
 \draw[edge] (0)--(010);
 \draw[edge] (0)--(010);
 \draw[edge,dotted] (010)--(110);
 \draw[edge] (0)--(001)--(011);
 \draw[edge] (0)--(001);
 \draw[edge] (100)--(101);
 \draw[edge, dotted] (100)--(110);
 
  \draw[-latex,line width=.5pt] (0)--(100);
 \draw[-latex,line width=.5pt] (001)--(101);
  \draw[-latex,line width=.5pt] (010)--(011);

 \end{tikzpicture}

(a)
\end{minipage}
 \caption{(a) A subcomplex $\cX$ of $\cH_3$. (b) The matching $\textsc{Mate}(\cdot,\balpha=(\alpha_1, \alpha_2, \alpha_3))$, with $\alpha_i$ given by \eqref{eqn:hypercube:alpha}, on $\cX$.
 }   \label{fig:subncubematch}
\end{figure}

\begin{proof} [Proof of Theorem~\ref{thm:hypercube}]
Let $\balpha=(\alpha_1,\alpha_2,\ldots,\alpha_n)$ where $\alpha_i$ is defined as in \eqref{eqn:hypercube:alpha}.  It follows from Proposition~\ref{prop:incidenceembed} that each $\alpha_i$ is an acyclic partial matching.  It follows  from Theorem~\ref{thm:algmatching} that $w=\Mate(\cdot,\balpha)$ is an $\balpha$-stable partial matching.  It remains to show that $w$ is acyclic. This will follow from the structure of the $n$-cube. Suppose that $w$ is not acyclic and that there is a sequence $\xi_0\gg \xi_1 \gg\ldots \gg \xi_n \gg \xi_0$ in $\cX$. This corresponds to a flowline
 \[
 \xi_0 \stackrel{i_0}{\prec} w(\xi_0) \succ x_1\stackrel{i_1}{\prec} w(\xi_1)\succ 
 \ldots \succ \xi_n\stackrel{i_n}{\prec}
 w(\xi_n)\succ \xi_0,
 \]
 where $\alpha_{i_k}(\xi_k) = w(\xi_k)$.
Set $\bx_i=\iota(\xi_i)$ and $\by_i=\iota(w(\xi_i))$.  Since $\iota$ is an incidence embedding there is a corresponding zig-zag path in the hypercube $\cH_n$
\[
\bx_0 \stackrel{i_0}{\prec} \by_0
\stackrel{j_0}{\succ} \bx_1 \stackrel{i_1}{\prec} \by_1
\stackrel{j_1}{\succ} \ldots
\stackrel{j_{n-1}}{\succ} \bx_n 
\stackrel{i_n}{\prec} \by_n
\stackrel{j_n}{\succ} \bx_0 .
\]
From the definition of $\balpha$, for each `zig', $\bx_k \stackrel{i_k}{\prec} \by_k$, we have that $\cM_{i_k}(\bx_k)=\by_k$, which corresponds to flipping the bit $x_{i_k}$ from 0 to 1.  Each `zag', $\by_k \succ \bx_{k+1}$, corresponds to flipping a bit from 1 to 0, thus may also be labeled with an appropriate coordinate $j_k$ where $\by_k = \cM_{j_k}(\bx_{k+1}).$

Let $i^*=\min\{i_0,\ldots,i_n\}$, i.e., the most preferred bit for the `zigs', and which corresponds to a flip of the bit $x_{i^*}$ from 0 to 1.   Since the zig-zag path is a loop, there must be a corresponding `zag' in the sequence through the $i^*$th bit, i.e., where the bit $x_{i^*}$ is flipped from 1 to 0.  Thus there is some $k$ where (modulo $n+1$)
\[
\bx_k \stackrel{i_k}{\prec} \by_k \stackrel{i^*}{\succ} \bx_{k+1} \stackrel{i_{k+1}}{\prec} \by_{k+1} .
\]
As the cells of a flowline are distinct, and $\iota$ is an injective map, we must have that $i^*< i_k, i_{k+1}$.  This in turn implies that $\xi_{k+1}\gg \xi_k$ is an $\balpha$-unstable pair, which contradicts the fact that $w$ is $\balpha$-stable.  Therefore there can be no such zig-zag path, implying that $w$ is acyclic.


\end{proof}

\begin{rem}
Given the acyclic partial matching $w(\cdot) = \Mate(\cdot, \balpha)$ there are algorithms which  determine the boundary operator $\partial_\bullet^{C(w)}$ and thus the chain complex $C_\bullet(w)$. Our implementation~\cite{cmcode} uses~\cite[Algorithm 3.12]{focm}.
\end{rem}

\begin{rem}
Let $f\colon (\cX,\leq) \to \sP$ be an order-preserving map. 
We can adapt the templates of \eqref{eqn:hypercube:alpha} to the graded case as follows.  Define $\beta_i\colon \cX\to \cX$ as
\[
\beta_i(\xi) =
\begin{cases}
  \alpha_i(\xi) & \text{if $f(\xi)=f(\alpha_i(\xi))$},\\
  \xi & \text{otherwise.}
\end{cases}
\]
As each $\alpha_i$ is an acyclic partial matching, it is straightforward that each $\beta_i$ is a $\sP$-graded acyclic partial matching.  It follows from Proposition~\ref{prop:gradedmate} that $\Mate(\cdot,\bbeta)$ is a $\sP$-graded partial matching, and it follows from Theorem~\ref{thm:hypercube} that $\Mate(\cdot,\bbeta)$ is acyclic.
\end{rem}

%% file: globalTemplate.tex
\section{Templates for Cell Complexes}\label{sec:general}

The goal of this section is to prove Theorem~\ref{thm:global}.  This theorem formalizes a procedure to decompose a cell complex $\cX$, each  part of which is $\cH_n$-embeddable for some $n$, and on which the hypercube templates and the \textsc{Mate} algorithm may be applied to yield a global acyclic partial matching. This can be done by organizing the decomposition as the fibers of an order-preserving map. The rest of the section demonstrates how this can be applied to the setting of cubical and simplicial complexes.

\begin{proof}[Proof of Theorem~\ref{thm:global}]
As each $\cX^p$ is $\cH_n$-embeddable for some $n$, it follows from Theorem~\ref{thm:hypercube} that for each $p\in \sP$, $\Mate(\cdot, \balpha_p)$ is an acyclic partial matching on $\cX^p$, where $\balpha_p$ is given in \eqref{eqn:hypercube:alpha}. It follows from  Theorem~\ref{thm:dmt:patchwork} that $w$ is an acyclic partial matching on $\cX$.
\end{proof}

\subsection{Templates for Cubical Complexes}\label{sec:general:cubical}

Let $\cC(m;d)$ be the cubical complex consisting of every elementary cube in $[0,m]^d\subset \R^d$.  The cubical complex $\cC(2;2)$ in $\R^2$ is depicted in Fig.~\ref{fig:cubical:2dcubical}(a).  

\begin{figure}[h!]
\begin{minipage}{.5\textwidth}
\centering
\begin{tikzpicture}[dot/.style={draw,circle,fill,inner sep=1.5pt},line width=.7pt]
\fill[gray!80,fill opacity=.2](0,0) rectangle (2,2);
\foreach \x in {0,1,2}
    \foreach \y in {0,1,2}
        \node (\x\y) at (\x,\y)[dot] {};
\foreach \x in {0,1,2}
    \foreach \y[count=\yi] in {0,1}
        \draw (\x\y)--(\x\yi);
\foreach \x [count=\xi]in {0,1}
    \foreach \y in {0,1,2}
        \draw (\x\y)--(\xi\y);
\end{tikzpicture}

(a)
\end{minipage}
\begin{minipage}{.5\textwidth}
\centering 
\begin{tikzpicture}[dot/.style={draw,circle,fill,inner sep=1.5pt},line width=.7pt]
\fill[gray!80,fill opacity=.2](0,0) rectangle (2,2);
\foreach \x in {0,1,2}
    \foreach \y in {0,1,2}
        \node (\x\y) at (\x,\y)[dot] {};
\foreach \x in {0,1,2}
    \foreach \y[count=\yi] in {0,1}
        \draw (\x\y)--(\x\yi);
\foreach \x [count=\xi]in {0,1}
    \foreach \y in {0,1,2}
        \draw (\x\y)--(\xi\y);
\foreach \y in {0,.5,1,1.5,2}
    \foreach \x in {0,1}
    \draw[-latex,line width=1.5pt] (\x,\y) -- (\x+.6,\y);
\foreach \y [count=\xi] in {0,2}
    \draw[-latex,line width=1.5pt] (1,\y) -- (1.6,\y);
\foreach \y [count=\xi] in {0,1}
    \draw[-latex,line width=1.5pt] (2,\y) -- (2,\y+.6);
\end{tikzpicture}

(b)
\end{minipage}
\caption{ (a) Cubical complex $\cC(2;2)$. (b) Acyclic partial matching $\Mate(\cdot,\balpha)$ applied to $\cC(2;2)$.}\label{fig:cubical:2dcubical}
\end{figure}

Any cube $\xi$ takes the form $\xi = I_1\times \ldots \times I_n$ where $I_i=[l_i,l_i]$ or $I_i = [l_i,l_i+1]$.   Let $\Z^d$ be the poset with order relation $(x_i)\leq (y_i)$ if and only if $x_i\leq y_i$ for all $i$. Define $f_d\colon \cC(m;d)\to \Z^d$ via
\[
f_d(\xi) = (-l_1,-l_2,\ldots,-l_d).
\]
It is elementary that $f_d$ is an order preserving map.   As indicated in Fig.~\ref{fig:cube:fiber}, a nonempty fiber $f_d^{-1}(p)$, 
where $p=(-l_1,-l_2,\ldots,-l_d)$, where $-l_i\neq -m$, is a set of cubes of the form
\[
[l_1,l_1+a_1]\times [l_2,l_2+a_2]\times \ldots \times [l_d,l_d+a_d] \quad \text{ where } a_i\in \{0,1\} .
\]
It is routine to check that the map $\iota \colon f^{-1}_d(p) \to \cH_d$ defined by 
\begin{equation}\label{eqn:cube:embed}
[l_1,l_1+a_1]\times [l_2,l_2+a_2]\times \ldots \times [l_d,l_d+a_d] \mapsto 
(a_1,a_2,\ldots, a_d)  \in \cH_d
\end{equation}
is an incidence embedding. If there is any $i$ such that $-l_i=-m$ then the corresponding elementary interval for any cell in the fiber is $I_i=[m,m]$, and there is an order-embedding into $\cH_{d-k}$, where $k$ is the total number of such intervals. Therefore the templates can be applied to each fiber via  \eqref{eqn:hypercube:alpha} and Theorem~\ref{thm:hypercube}, see Fig.~\ref{fig:cube:fiber}.   By Theorem~\ref{thm:global}, these  assemble into a global matching, $w\colon C(m;d)\to C(m;d)$, see Fig.~\ref{fig:cubical:2dcubical} (b).

\begin{figure}[h!]
\begin{minipage}[t]{0.48\textwidth}
\centering
\begin{tikzpicture}
 \tikzstyle{vertex}=[circle,inner sep=1pt]
  \tikzstyle{vert}=[circle,fill,inner sep=1.5pt]
\tikzstyle{edge} = [draw,thick,-,black]
 \tikzstyle{medge} = [draw,densely dashdotted,black]
\def\x{2.5}
\def\y{.75}
\def\z{(.5*\x+.5*\y}
\def\a{\x+0.05}
\def\b{\y+0.05}
 \coordinate[vert] (O) at (0,0,0) {};

\draw[-latex,line width=1.5pt] (0,0,0) -- (\y,0,0);
\draw[-latex,line width=1.5pt] (0,0,-\z) -- (\y,0,-\z);
\draw[-latex,line width=1.5pt] (0,\z,0) -- (\y,\z,0);
\draw[-latex,line width=1.5pt] (0,\z,-\z) -- (\y,\z,-\z);
 
 \coordinate (x1) at (\y,0,0) {};
 \coordinate (x2) at (\x,0,0) {};
  \draw[edge] (x1) -- (x2);

 \coordinate (y1) at (0,0,-\y) {};
 \coordinate (y2) at (0,0,-\x) {};
   \draw[edge] (y1) -- (y2);

 \coordinate (z1) at (0,\y,0) {};
 \coordinate (z2) at (0,\x,0) {};
 \draw[edge] (z1) -- (z2);
 
  \coordinate (bbb) at (\b,\b,-\b) {};
 \coordinate (bba) at (\b,\b,-\a) {};
 \coordinate (bab) at (\b,\a,-\b) {};
 \coordinate (baa) at (\b,\a,-\a) {};
 \coordinate (abb) at (\a,\b,-\b) {};
 \coordinate (aba) at (\a,\b,-\a) {};
 \coordinate (aab) at (\a,\a,-\b) {};
 \coordinate (aaa) at (\a,\a,-\a) {};
  \draw[medge] (bbb) -- (bba) -- (baa) -- (bab) -- cycle;
  \draw[medge] (bbb) -- (abb) -- (aba) -- (bba) -- cycle;
\draw[medge] (bbb) -- (abb) -- (aab) -- (bab) -- cycle;
\draw[medge] (abb) -- (aba) -- (aaa) -- (aab) -- cycle;
\draw[medge] (baa) -- (aaa);

   \fill [gray!10,fill opacity = .2] (bba)--(aba)--(aaa)--(baa)--cycle;
 \fill [gray!20,fill opacity = .2] (bbb) -- (bba) -- (baa) -- (bab) -- cycle;
  \fill [gray!20,fill opacity = .2] (bbb) -- (abb) -- (aba) -- (bba) -- cycle;
  \fill [gray!20,fill opacity = .2] (bbb) -- (abb) -- (aab) -- (bab) -- cycle;
 \fill [gray!20,fill opacity = .2] (abb) -- (aba) -- (aaa) -- (aab) -- cycle;
  \fill [gray!20,fill opacity = .2] (bba)--(aba)--(aaa)--(baa)--cycle;

 \coordinate (xz1) at (\y,\y,0) {};
 \coordinate (xz2) at (\x,\y,0) {};
 \coordinate (xz3) at (\x,\x,0) {};
 \coordinate (xz4) at (\y,\x,0) {};
 \draw[medge] (xz1) -- (xz2) -- (xz3) -- (xz4) -- cycle;
 \fill [gray!50,fill opacity = .5](xz1) -- (xz2) -- (xz3) -- (xz4) -- cycle;

 \coordinate (xy1) at (\y,0,-\y) {};
 \coordinate (xy2) at (\x,0,-\y) {};
 \coordinate (xy3) at (\x,0,-\x) {};
 \coordinate (xy4) at (\y,0,-\x) {};
 \draw[medge] (xy1) -- (xy2) -- (xy3) -- (xy4) -- cycle;
  \fill [gray!50,fill opacity = .5] (xy1) -- (xy2) -- (xy3) -- (xy4) -- cycle;

 \coordinate (yz1) at (0,\y,-\y) {};
 \coordinate (yz2) at (0,\y,-\x) {};
 \coordinate (yz3) at (0,\x,-\x) {};
 \coordinate (yz4) at (0,\x,-\y) {};
 \draw[medge] (yz1) -- (yz2) -- (yz3) -- (yz4) -- cycle;
  \fill [gray!50,fill opacity = .5] (yz1) -- (yz2) -- (yz3) -- (yz4) -- cycle;

\end{tikzpicture}

(a)
\end{minipage}
\hspace{\fill}
\begin{minipage}[t]{0.48\textwidth}
\centering
\begin{tikzpicture}
 \tikzstyle{vertex}=[circle,minimum size=20pt,inner sep=0pt]
 \tikzstyle{edge} = [draw,thick,-,black]
  \tikzstyle{medge} = [draw,dotted,black]
\def\x{1.5}
\def\y{1}
\node[vertex] (0) at (0,0) {$000$};

\node[vertex] (100) at (\x,\y) {$100$};
\node[vertex] (001) at (0,\y) {$001$};
\node[vertex] (010) at (-\x,\y) {$010$};

\node[vertex] (110) at (0,2*\y) {$110$};
 \node[vertex] (011) at (-\x,2*\y) {$011$};
 \node[vertex] (101) at (\x,2*\y) {$101$};

 \node[vertex] (111) at (0,3*\y) {$111$};
 \draw[edge] (0) -- (001);
 \draw[edge] (0)--(010);
 \draw[edge] (001)--(011);
  \draw[edge] (010)--(011);
 \draw[edge] (100)--(101);
  \draw[edge] (100)--(110);
 \draw[edge] (110)--(111);
  \draw[edge] (101)--(111);
 \draw[-latex,line width=.5pt] (0)--(100);
 \draw[-latex,,line width=.5pt] (001)--(101);
 \draw[-latex,,line width=.5pt] (011)--(111);
\draw[-latex,,line width=.5pt] (010)--(110);
\end{tikzpicture}

(b)
\end{minipage}

\caption{(a) Example of fiber $f_3^{-1}(p)$ and template $\alpha_1$ in cubical complex. (b) Hasse diagram of the face poset of $\cH_3$ with template $\cM_1$.}\label{fig:cube:fiber}
\end{figure}
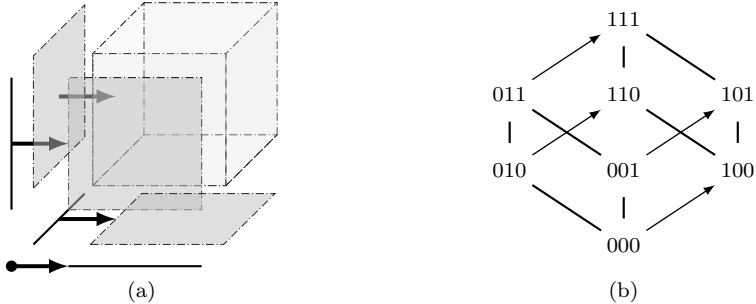


\begin{rem}
Let $\cX\subset \cC(m;d)$ in $\R^d$ be an arbitrary cubical complex and consider $f_d|_\cX\colon \cX\to \Z^d$.  
A fiber of $f_d|_\cX$ is embeddable into $\cH_d$ via a restriction of \eqref{eqn:cube:embed}.  
From Theorem~\ref{thm:global} it follows that the templates can be used to assemble a global matching $w\colon \cX\to \cX$.
\end{rem}

\subsection{Templates for Simplicial Complexes}

We outline how the templates can be used for simplicial complexes.  In fact, they are related to well-known acyclic partial matchings on intervals in the face poset of a simplicial complex~\cite{bauer2017morse,freij2009equivariant}, which~\cite{bauer2017morse} call \emph{vertex refinements}. See Fig.~\ref{fig:simplicial:hypercube}. 

Let $\cK^n$ be an $n$-simplex with vertex set $V=\{v_1,\ldots, v_{n+1}\}$. A simplex $\xi\in \cK^n$ is a subset of $V$.  There is an order embedding $\iota_n\colon \cK^n\to \cH_{n+1}$ given by 
\[
\xi\mapsto 
\bx=(x_i) \text{ where } \begin{cases}
  x_i = 1 & \text{if } v_i\in \xi,\\
  x_i = 0 & \text{otherwise.}
\end{cases}
\]

\begin{figure}[h!]
\centering
\begin{minipage}{.45\textwidth}
\centering
\begin{tikzpicture}
  \tikzstyle{vertex}=[circle,inner sep=1pt]
  \tikzstyle{vert}=[circle,fill,inner sep=1.5pt]
  \def\x{0}
    \def\dx{1}
  \def\y{1.35}
\tikzstyle{edge} = [draw,-,black]
 \tikzstyle{medge} = [draw,densely dotted,black]
 \coordinate[vert,label=left:$v_1$] (v0) at (0,0) {};
\coordinate[vert,label=above:$v_3$] (v2) at (0,\y) {};
\coordinate[vert,label=right:$v_2$] (v1) at (1,1) {};
\node (0) at (0,0) {};
\node (2) at (0,\y) {};
\node (3) at (1,1) {};
\node (02) at (0, .5) {};
\node (03) at (.5, .5) {};

\node (11) at (.5, 1.3) {};
\node (22) at (.20, .5) {};

 \draw[edge] (0)--(2) node (02) [midway] {};
 \draw[edge] (0)--(3) node (03) [midway] {};
\draw[edge] (2)--(3) node (23) [midway] {};
\fill [gray!80,fill opacity=0.2] (0,0)--(1,1)--(0,\y)--cycle;

\draw[-latex,line width=1.5pt] (2) -- (02);
\draw[-latex,line width=1.5pt] (3) -- (03);
\draw[-latex,line width=1.5pt] (11) -- (22);

\end{tikzpicture}

\end{minipage}
\begin{minipage}{.5\textwidth}
    \centering
    \begin{tikzpicture}
 \tikzstyle{vertex}=[circle,minimum size=20pt,inner sep=0pt]
 \tikzstyle{edge} = [draw,thick,-,black]
  \tikzstyle{medge} = [draw,dotted,black]
\def\x{1.85}
\def\y{1.1}
\node (0) at (0,0) {};

\node (100) at (\x,\y) {$\{v_1\}$};
\node (001) at (0,\y) {$\{v_3\}$};
\node (010) at (-\x,\y) {$\{v_2\}$};

\node(110) at (0,2*\y) {$\{v_1,v_2\}$};
 \node (011) at (-\x,2*\y) {$\{v_2,v_3\}$};
 \node (101) at (\x,2*\y) {$\{v_1,v_3\}$};
 
 \node (111) at (0,3*\y) {$\{v_1,v_2,v_3\}$};
 
 \draw[edge] (001)--(011);
  \draw[edge] (010)--(011);
 \draw[edge] (100)--(101);
  \draw[edge] (100)--(110);
 \draw[edge] (110)--(111);
  \draw[edge] (101)--(111);
 \draw[-latex,,line width=.5pt] (001)--(101);
 \draw[-latex,,line width=.5pt] (011)--(111);
\draw[-latex,,line width=.5pt] (010)--(110);
\end{tikzpicture}

\end{minipage}

(a) \hspace{55mm} (b)
\caption{(a) The 2-simplex $\cK^2$ and matching  $\alpha_1$, called the vertex refinement (along $v_1$). (b) The Hasse diagram of the face poset of $\cK^2$ and matching $\alpha_1$, cf. Fig.~\ref{fig:cube:fiber}.}\label{fig:simplicial:hypercube}
\end{figure}

 More generally, let $\cK$ be a simplicial complex.  Recall that given $\xi,\xi'\in \cK$ with $\xi' \leq \xi$  the {\em interval} from $\xi'$ to $\xi$, denoted $[\xi',\xi]$, is the set $\{\xi''\in \cK\mid  \xi' \leq \xi'' \leq \xi\}$. The order embedding $\iota_n$ restricts to a order embedding $\iota_n|_{[\xi',\xi]}\colon [\xi',\xi] \to \cH_{n+1}$, where $n=\dim (\xi)$. Again, through \eqref{eqn:hypercube:alpha} and Theorem~\ref{thm:hypercube} the hypercube templates can be pulled back to $[\xi',\xi]$ as in Fig.~\ref{fig:simplicial:hypercube}.
The strategy taken in~\cite{bauer2017morse,freij2009equivariant} to find an acyclic partial matching on $\cK$ is to find a {\em generalized discrete Morse function}: an order-preserving map $f\colon \cK\to \R$ where each fiber is an  interval, and on each non-singleton interval use a vertex refinement along some $v\in \xi\setminus \xi'$; cf. Theorem~\ref{thm:global}.


%% file: experiments.tex
\section{Computational Experiments}\label{sec:experiments}

We provide two sets of computational experiments for demonstrating the efficacy of the template scheme, as well as for  benchmarking the performance of our implementation of $\Mate$/$\GMate$ within our software package \textsc{PyCHomP}~\cite{cmcode}.\footnote{Jupyter notebooks for all computational experiments found in this section are available at \url{https://github.com/kellyspendlove/pyCHomP/tree/cubical/doc/Templates-Paper}.}  Our first set of examples are cubical complexes which are homotopy equivalent to $d$-spheres.    Our second set of experiments are drawn from an application of computational connection matrix theory~\cite{hms} to the Conley-Morse theory on spaces of braid diagrams developed in~\cite{im}.   

We track timing information, which comes from iPython's \%timeit `magic', and displays an average and standard deviation of seven runs for cell complexes with less than $10^9$ cells, and two runs for cell complexes with more.  All experiments were performed with 1 core on a machine with 72 Intel(R) Xeon(R) Gold 6240M  2.60GHz CPUs and 3TB of RAM. 

\subsection{The Homology of Cubical $d$-Spheres}\label{sec:experiments:hypersphere}

We give two different examples of cubical complexes homotopy equivalent to $d$-spheres.   As suggested by Fig.~\ref{fig:S} define
\[
\cS^d = \{\xi\in  \cC(1;d+1)\mid \dim \xi < d+1 \}
\quad \text{and} \quad \cS_{top}^d = \{\xi\in  \cC(3;d+1)\mid \xi \neq [1,2]^d \}.
\]

\begin{figure}[h!]
\begin{minipage}{.5\textwidth}
\centering
\begin{tikzpicture}[dot/.style={draw,circle,fill,inner sep=1.5pt},line width=.7pt]
\foreach \x in {0,1}
    \foreach \y in {0,1}
        \node (\x\y) at (\x,\y)[dot] {};
\foreach \x in {0,1}
    \foreach \y[count=\yi] in {0}
        \draw (\x\y)--(\x\yi);
\foreach \x [count=\xi]in {0}
    \foreach \y in {0,1}
        \draw (\x\y)--(\xi\y);
\end{tikzpicture}
\end{minipage}
\begin{minipage}{.5\textwidth}
\centering
\begin{tikzpicture}[dot/.style={draw,circle,fill,inner sep=1pt},line width=.7pt,scale=.75]
\fill[gray!80,fill opacity=.2](0,0) rectangle (3,3);
\fill[white](1,1) rectangle (2,2);
\foreach \x in {0,1,2,3}
    \foreach \y in {0,1,2,3}
        \node (\x\y) at (\x,\y)[dot] {};

\foreach \x in {0,1,2,3}
    \foreach \y[count=\yi] in {0,1,2}
        \draw (\x\y)--(\x\yi);
\foreach \x [count=\xi]in {0,1,2}
    \foreach \y in {0,1,2,3}
        \draw (\x\y)--(\xi\y);

\end{tikzpicture}
\end{minipage}

\centering
(a) \hspace{55mm} (b)

\caption{ Cubical complexes (a) $\cS^1\subset \R^2$ and (b) $\cS_{top}^1\subset \R^{2}$.}
\label{fig:S}
\end{figure}
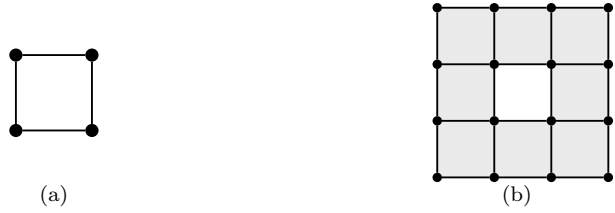

In order to compute the homology of these cubical complexes, we utilize the algorithm $\HOM$~\cite[Algorithm 6.2]{hms}, which is based on iterated discrete Morse theory, see~\cite{hms,mn}.  Iterated discrete Morse theory performs rounds of discrete Morse theory until the boundary operator is zero, i.e., $\partial_\bullet^{C(w)}=0.$
In this case, the $n$-th Betti number of $\cX$ is $\dim C_n(w)$, the dimension of the vector space $C_n(w)$. For $\HOM$, the use of the templates and $\Mate$ applies only during the first round of discrete Morse theory as the subsequent complexes are not necessarily cubical or $\cH_n$-embeddable; our implementation uses~\cite[Algorithm 3.6]{focm} for the latter rounds of discrete Morse theory, as it applies to generic cell complexes.   For more detail, see the discussion around~\cite[Algorithm 6.2]{hms}.  However, note that in this particular case, when applying $\HOM$ to $\cS^d$ and $\cS_{top}^d$, only one round of Morse theory is needed, as $\Mate(\cdot,\balpha=(\alpha_i))$, with $\alpha_i$ defined as in Section~\ref{sec:general:cubical}  will produce an optimal matching, and thus in this case a chain complex obeying $\partial_\bullet^{C(w)}=0$.  Table~\ref{tab:hypsphere} gives the results of the experiments and Fig.~\ref{plot:intro} plots the results of Table~\ref{tab:hypsphere}.

From the tables it can be seen that the algorithm $\HOM$ scales roughly linearly with $N$, the number of cells.  

\begin{table}[h!]
\centering
\begin{subtable}[t]{0.48\textwidth}
\setlength{\tabcolsep}{5pt}
\centering
\begin{tabular}[t]{l l l} 
\toprule
   \multicolumn{3}{c @{}}{Experimental Results for $\cS^d$ }\\ 
\cmidrule(l){1-3}
  {$d$}  & {$\#\cS^d$ } &  { Time Elapsed }  \\
\midrule
5 & 242 &  $160\mu s \pm 1.65 \mu s$   \\
6  & 728 &   $384 \mu s \pm 823 n s$  \\
7  & 2186 &  $998 \mu s \pm 17 \mu s$  \\
8 & 6560 &   $2.84 m s \pm 31.3 \mu s$    \\
9 & 19682 & $8.07 ms \pm 139 \mu s$  \\
10 & 59048 &   $24.5 m s \pm 102 \mu s$  \\
11  & 177146 &  $74.1 m s \pm 186 \mu s $   \\
12  & 531440 &  $223 ms \pm 1.96 ms$  \\
13 & 1594322 &    $647 ms \pm 12 ms$   \\
14 & 4782968 &  $2.05 s \pm 13.2 ms$ \\
15 & 14348906 &  $6.27 s \pm 103 ms$   \\
16  & 43046720 &   $18.7 s \pm 148 ms$  \\
17  & 129140162 & $53.7 s \pm 21.2 ms$  \\
18 & 387420488 &  $2m42s \pm 64.4 ms$\\
19 & 1162261466 &  $8m54s\pm 1.96s$ \\
20 & 3486784400 &  $27m 12s \pm 6.24 s$ \\
\bottomrule
\end{tabular}
\label{tab:hypersphere}
\end{subtable}
\hspace{\fill}
\begin{subtable}[t]{.48\textwidth}
\flushright
\setlength{\tabcolsep}{5pt}
\centering
\begin{tabular}[t]{l l l} 
\toprule
   \multicolumn{3}{c @{}}{Experimental Results for $\cS_{top}^d$ }\\ 
\cmidrule(l){1-3}
  {$d$}  & {$\#\cS^d_{top}$ } &  { Time Elapsed }  \\
\midrule
3 & 342 & $325 \mu s \pm 929 ns$\\
4 & 2400 & $1.98 ms \pm 1.63 \mu s$  \\
5 & 16806 &   $13.6 ms \pm 128 \mu s$  \\
6 & 117648 &  $92.3 ms \pm 87.7 \mu s$   \\
7 & 823542 &  $656 ms \pm 934 \mu s$  \\
8 & 5764800 & $4.72 s \pm 1.23 ms$ \\
9 & 40353606 & $34.3 s \pm 27.6 ms$  \\
10 & 282475248  & $4m14s\pm 841 ms$    \\
11 & 1977326742 & $31m18s \pm 7.12 s$\\
12 & 13841287200 & $3h55m 41s\pm 23.8 s$\\
\bottomrule
\end{tabular}
\end{subtable}
(a)\hspace{60mm} (b)
\caption{Results for $\HOM$ on (a)  $\cS^d$ and (b)  $\cS_{top}^d$.}\label{tab:hypsphere}
\end{table}

\begin{rem}
We note that the templates scheme and the algorithm $\HOM$ implemented in the package \textsc{PyCHomP} have also been used to compute the homology of high dimensional complexes in~\cite{config}.
\end{rem}


\subsection{A Conley-Morse Theory on the Space of Braid Diagrams}\label{sec:braids}

Our second set of experiments is drawn from an application of computational connection matrix theory~\cite{hms} to a Conley-Morse theory on spaces of braid diagrams~\cite{im}.  
See~\cite{braids} for a more detailed discussion of similar examples and their Conley complexes/connection matrices in the context of dynamics and scalar parabolic PDEs which are produced by~\cite{cmcode}. 

The Conley-Morse theory on braid diagrams provides a nice set of scalable and relevant graded cubical complexes which we will use for benchmarking the template scheme and $\GMate$ algorithm.  In particular, we compute the Conley complexes/connection matrices for these graded cubical complexes using the algorithm~\cite[Algorithm 6.8]{hms}, $\CM$.  As we are primarily interested in the use of templates, we  provide only a brief description of the theory and setup.  See~\cite{hms} for a detailed discussion of the algorithm and the connection matrix theory, and~\cite{braids} for interpretation of these results in terms of dynamics.


\begin{defn}\label{defn:discbraid}
{\em
The space of {\em period $d$ discrete braid diagrams on $m$ strands}, denoted $\Conf_d^m$, is the space of all pairs $(\uu,\tau)$ where $\tau\in S_m$ is a permutation on $m$ elements and $\uu=\{u^\alpha\}_{\alpha=1}^m$ is an unordered collection of vectors -- called {\em strands} -- each individual strand $u^\alpha = (u_i^\alpha)_{i=1}^{d+1}\in \R^{d+1}$, satisfying the following conditions:
\begin{enumerate}
    \item {\bf Periodicity:} $u_{d+1}^\alpha = u_1^{\tau(\alpha)}$ for all $\alpha$.
    \item {\bf Transversality:} for any $\alpha\neq \alpha'$ such that $u_i^\alpha = u_i^{\alpha'}$ for some $i$,
    \begin{equation}
        (u_{i-1}^\alpha-u_{i-1}^{\alpha'})(u_{i+1}^\alpha - u_{i+1}^{\alpha'}) < 0.
    \end{equation}
\end{enumerate}
}
\end{defn}

\begin{figure}[h!]
\begin{minipage}{.5\textwidth}
\centering
\begin{tikzpicture}[dot/.style={draw,circle,fill,inner sep=.75pt},line width=.5pt,scale=.5]
\foreach \x in {0,1,2}
    \foreach \y in {0,.75,1.5,2.5,3.25,4}
        \node (\x\y) at (\x, \y)[dot] {};
\foreach \y in {0,4}
    \foreach \x[count=\xi]  in {0,1}
        \draw (\x\y) to (\xi\y);
\foreach \x[count=\xi] in {0} {
     \draw (\x, .75) to (\xi, 2.5);
     \draw (\x, 1.5) to (\xi, .75);
     \draw (\x, 2.5) to (\xi, 3.25);
     \draw (\x, 3.25) to (\xi, 1.5); }
\foreach \x in {1} {
     \draw (\x, .75) to (\x+1, 1.5);
     \draw (\x, 1.5) to (\x+1, 3.25);
     \draw (\x, 2.5) to (\x+1, .75);
     \draw (\x, 3.25) to (\x+1, 2.5); }
\end{tikzpicture}
\end{minipage}
\begin{minipage}{.5\textwidth}
\centering
\begin{tikzpicture}[dot/.style={draw,circle,fill,inner sep=.75pt},line width=.5pt,scale=.5]
\foreach \x in {0,1,2}
    \foreach \y in {0,.75,1.5,2.5,3.25,4}
        \node (\x\y) at (\x, \y)[dot] {};
\node (f1) at (0,2.15)[dot] {};
\node (f2) at (1,1.05)[dot] {};
\node (f3) at (2,2.15)[dot] {};
\draw[dashed] (f1) -- (f2) -- (f3);
\node (f0) at (-.5,2.15) {$\uu$};
\foreach \y in {0,4}
    \foreach \x[count=\xi]  in {0,1}
        \draw (\x\y) to (\xi\y);
\foreach \x[count=\xi] in {0} {
     \draw (\x, .75) to (\xi, 2.5);
     \draw (\x, 1.5) to (\xi, .75);
     \draw (\x, 2.5) to (\xi, 3.25);
     \draw (\x, 3.25) to (\xi, 1.5); }
\foreach \x in {1} {
     \draw (\x, .75) to (\x+1, 1.5);
     \draw (\x, 1.5) to (\x+1, 3.25);
     \draw (\x, 2.5) to (\x+1, .75);
     \draw (\x, 3.25) to (\x+1, 2.5); }
\end{tikzpicture}
\end{minipage}
\caption{(a) Braid diagram $\vv\in \Conf_2^6$. (b) Free strand $\uu$ is dashed, $\uu\sqcup \vv\in \Conf_2^7$.}\label{fig:braid}

\end{figure}

Since the strands of $\uu$ are unordered, all pairs $(\uu,\tau)$ and $(\uu,\tau')$ satisfying $\tau'=\sigma\tau\sigma^{-1}$ for some $\sigma\in S_m$ are identified.  Henceforth, we suppress  the permutation $\tau$ from the notation.  We will also refer to $\uu$ simply as a \emph{braid diagram}.

For the purposes of this paper, a \emph{relative (discrete) braid diagram}, denoted $\uu\rel \vv$, consists of a single strand $\uu\in \Conf_d^1$ -- called the \emph{free strand} -- and  strands $\vv \in \Conf_d^m$ -- called the {\em skeleton} -- such that the disjoint union $\uu\sqcup \vv$
is a braid diagram in $\Conf_d^{m+1}$ according to the above definition. Following~\cite{im}, for fixed $\vv\in \Conf_d^m$, we define 
\[
\Conf_d^1\rel \vv = \{\uu\in \Conf_d^1 : \uu\sqcup \vv\in \Conf_d^{m+1}\} \subset \R^d.
\]
The path components of $\Conf_d^1\rel \vv$ are called {\em relative (discrete) braid classes}, and are denoted by $[\uu\rel\vv]$.

Given a skeleton $\vv\in \Conf_d^m$, the two parameters -- the dimension $d$ and the number of strands $m$ --  suffice to determine the cubical complex $\cX\subseteq \R^d$~\cite{im}.  We make two assumptions on our skeletons $\vv$ in order to construct $\cX$ as the complex $\cC=\cC(m-1;d)$.
\begin{enumerate}
    \item First, and without loss of generality (up to scaling), we assume the following condition on our skeleton $\vv$: for any fixed $i$ with $1\leq i \leq d+1$ the cross-section  $(v_i^1,v_i^2,\ldots,v_i^m)$ of the strands is a permutation of $(0,1,\ldots,m-1)$.  That is, the $(v_i^\alpha)$ are integers and take unique values between 0 and $m-1$.  Our condition on $\vv$ implies that given a tuple $(i,v_i^\alpha)$ we have that $(i,v_i^\alpha)\in \Z\times \Z$, and in particular $(i,v_i^\alpha)$ lies on the integer lattice within the rectangle $[0,d]\times [0,m-1]$.  
    \item Second, we assume that our skeleton $\vv\in \Conf_d^m$ contains constant strands $v^\alpha = (0,0,\ldots,0)$ and $v^{\alpha'} = (m-1,m-1,\ldots,m-1)$; see Fig.~\ref{fig:braid}.  This enforces boundary conditions, see~\cite{im,braids}. 
\end{enumerate}

The relative braid classes $[\uu\rel \vv]$ are precisely the interiors of the top dimensional cubes -- the $d$-dimensional cubes $\cC_d$ -- of $\cC$.  A free strand $\uu$ and skeleton  $\vv$ have an associated crossing number, denoted $\cross(\uu\rel\vv)$, which is the number of intersections of the strand $\uu$ with the strands in $\vv$.  For $\uu'\in [\uu\rel\vv]$ we have $\cross(\uu'\rel\vv)=\cross(\uu\rel\vv)$, see~\cite{im}. Therefore the crossing number is constant on $d$-cubes $\cC_d$.  This furnishes a function
\[
\cross_\vv \colon \cC_d \to \N.
\]
The function $\cross_\vv$ is a discrete Lyapunov function for any parabolic flow which fixes $\vv$ as a set of equilibrium solutions; see~\cite{im} for more details.   In Fig.~\ref{fig:lap}(a) is the cubical complex $\cC=\cC(5;2)$ associated to the skeleton $\vv\in \Conf_2^6$ in Fig.~\ref{fig:braid}(a); the integer labeling on the $2$-cubes $\cC_2$ gives the value of $\cross_\vv(\cdot)$ on each $2$-cube.  

\begin{figure}[h!]
\begin{minipage}{.45\textwidth}
\centering
\begin{tikzpicture}[dot/.style={draw,circle,fill,inner sep=.75pt}, scale=.55]
\draw[step=1cm,gray,very thin] (0,0) grid (5,5);
\foreach \x in {.5} {
    \foreach \y [evaluate = \y as \z using int(2*(\y-.5))]  in {.5, 1.5, 2.5, 3.5, 4.5} {
        \node (\x\y) at (\x, \y) {\z}; }
    \foreach \y [evaluate = \y as \z using int((8-2*(\y-.5)))]  in {.5, 1.5, 2.5, 3.5, 4.5} {
        \node (\x\y) at (\x+4, \y) {\z}; }
}
\foreach \y in {.5} {
     \foreach \x [evaluate = \x as \z using int(2*(\x-.5))]  in {1.5, 2.5, 3.5} {
        \node (\x\y) at (\x, \y) {\z};  }
    \foreach \x [evaluate = \x as \z using int(8-2*(\x-.5))]  in {1.5, 2.5, 3.5} {
        \node (\x\y) at (\x, \y+4) {\z}; }
}

\node (0) at (1.5,1.5) {$4$}; 
\node (1) at (2.5,1.5) {$2$}; 
\node (2) at (3.5,1.5) {$4$}; 

\node (3) at (1.5,2.5) {$6$}; 
\node (4) at (2.5,2.5) {$4$}; 
\node (5) at (3.5,2.5) {$6$}; 

\node (6) at (1.5,3.5) {$4$}; 
\node (7) at (2.5,3.5) {$2$}; 
\node (8) at (3.5,3.5) {$4$}; 
\draw[thick, -stealth] (0,0) -- (5.1,0);
\draw[thick,-stealth] (0,0) -- (0,5.1);


\node (d0) at (0,0)[dot] {};
\node (d1) at (1,3)[dot] {};
\node (d2) at (2,1)[dot] {};
\node (d3) at (3,4)[dot] {};
\node (d4) at (4,2)[dot] {};
\node (d5) at (5,5)[dot] {};

\node (u) at (2.75,1.3)[dot] {};
\node (8) at (2.85,1.05) {\scriptsize $\uu$};

\end{tikzpicture}

(a)
\end{minipage}
\begin{minipage}{.45\textwidth}
\centering
\begin{tikzpicture}[dot/.style={draw,circle,fill,inner sep=.75pt}, scale=.55]
\draw[step=1cm,gray,very thin] (0,0) grid (5,5);
\foreach \x in {.5} {
    \foreach \y [evaluate = \y as \z using int(2*(\y-.5))]  in {.5, 1.5, 2.5, 3.5, 4.5} {
        \node (\x\y) at (\x, \y) {\z}; }
    \foreach \y [evaluate = \y as \z using int((8-2*(\y-.5)))]  in {.5, 1.5, 2.5, 3.5, 4.5} {
        \node (\x\y) at (\x+4, \y) {\z}; }
}
\foreach \y in {.5} {
     \foreach \x [evaluate = \x as \z using int(2*(\x-.5))]  in {1.5, 2.5, 3.5} {
        \node (\x\y) at (\x, \y) {\z};  }
    \foreach \x [evaluate = \x as \z using int(8-2*(\x-.5))]  in {1.5, 2.5, 3.5} {
        \node (\x\y) at (\x, \y+4) {\z}; }
}
\foreach \x in {.5} 
    \foreach \y[count=\yi]  in {4.25, 3.25, 2.25, 1.25} 
        \draw[-latex,line width=.5pt] (\x,\y) to (\x,\y-.55);
\foreach \x in {1.5} 
    \foreach \y[count=\yi]  in {4.25, 3.25, 2.25, 1.25} 
        \draw[-latex,line width=.5pt] (\x,\y) to (\x,\y-.55); 
\foreach \x in {4.5} 
    \foreach \y[count=\yi]  in {.75, 1.75, 2.75, 3.75} 
        \draw[-latex,line width=.5pt] (\x,\y) to (\x,\y+.55);
\foreach \x in {3.5} 
    \foreach \y[count=\yi]  in {.75, 1.75, 2.75, 3.75} 
        \draw[-latex,line width=.5pt] (\x,\y) to (\x,\y+.55);
\foreach \x in {2.5} {
    \foreach \y in {1.5, 2.5, 4.5} {
        \draw[-latex,line width=.5pt] (\x,\y-.25) to (\x,\y-.8);
    }
      \foreach \y in {.5, 2.5, 3.5} {
        \draw[-latex,line width=.5pt] (\x,\y+.25) to (\x,\y+.8);
    }
}    
\foreach \y in {4.5} 
    \foreach \x[count=\yi]  in {.75, 1.75, 2.75, 3.75} 
    \draw[-latex,line width=.5pt] (\x,\y) to (\x+.55,\y);
\foreach \y in {3.5} 
    \foreach \x[count=\yi]  in {.75, 1.75, 2.75, 3.75} 
    \draw[-latex,line width=.5pt] (\x,\y) to (\x+.55,\y);
\foreach \y in {.5} 
    \foreach \x[count=\yi]  in {4.25, 3.25, 2.25, 1.25} 
    \draw[-latex,line width=.5pt] (\x,\y) to (\x-.55,\y);
\foreach \y in {1.5} 
    \foreach \x[count=\yi]  in {4.25, 3.25, 2.25, 1.25} 
    \draw[-latex,line width=.5pt] (\x,\y) to (\x-.55,\y);
    
\foreach \y in {2.5} {
    \foreach \x  in {.5, 1.5, 3.5} {
        \draw[-latex,line width=.5pt] (\x+.25,\y) to (\x+.8,\y);
        }
    \foreach \x  in {1.5, 3.5,4.5} {
        \draw[-latex,line width=.5pt] (\x-.25,\y) to (\x-.8,\y);
    }
}
\draw[-latex,line width=.5pt] (.5,2.5+.25) to (.5,2.5+.8);
\draw[-latex,line width=.5pt] (1.5,2.5+.25) to (1.5,2.5+.8);
\draw[-latex,line width=.5pt] (1.5,0.5+.25) to (1.5,0.5+.8);
\draw[-latex,line width=.5pt] (3.5,2.5-.25) to (3.5,2.5-.8);
\draw[-latex,line width=.5pt] (3.5,4.5-.25) to (3.5,4.5-.8);
\draw[-latex,line width=.5pt] (4.5,2.5-.25) to (4.5,2.5-.8);

\draw[-latex,line width=.5pt] (1.5+.25,0.5) to (1.5+.8,0.5);
\draw[-latex,line width=.5pt] (1.5+.25,1.5) to (1.5+.8,1.5);
\draw[-latex,line width=.5pt] (3.5+.25,1.5) to (3.5+.8,1.5);

\draw[-latex,line width=.5pt] (3.5-.25,4.5) to (3.5-.8,4.5);
\draw[-latex,line width=.5pt] (3.5-.25,3.5) to (3.5-.8,3.5);
\draw[-latex,line width=.5pt] (1.5-.25,3.5) to (1.5-.8,3.5);

        
 

\node (0) at (1.5,1.5) {$4$}; 
\node (1) at (2.5,1.5) {$2$}; 
\node (2) at (3.5,1.5) {$4$}; 

\node (3) at (1.5,2.5) {$6$}; 
\node (4) at (2.5,2.5) {$4$}; 
\node (5) at (3.5,2.5) {$6$}; 

\node (6) at (1.5,3.5) {$4$}; 
\node (7) at (2.5,3.5) {$2$}; 
\node (8) at (3.5,3.5) {$4$}; 
\draw[thick, -stealth] (0,0) -- (5.1,0);
\draw[thick,-stealth] (0,0) -- (0,5.1);


\node (d0) at (0,0)[dot] {};
\node (d1) at (1,3)[dot] {};
\node (d2) at (2,1)[dot] {};
\node (d3) at (3,4)[dot] {};
\node (d4) at (4,2)[dot] {};
\node (d5) at (5,5)[dot] {};

\end{tikzpicture}

(b)
\end{minipage}
\caption{(a) Cubical complex $\cC(5;2)$, $\cross_\vv\colon \cC(5;2)_2\to \N$ and free strand $\uu$, where $\vv$ and $\uu$ are given in Fig.~\ref{fig:braid}.  Note $\cross(\uu\rel \vv) = 2$. The other points are the improper vertices in $G_\vv$.  (b) Depiction of binary relation $\cF$ as a directed graph.} \label{fig:lap}
\end{figure}


Following~\cite{braids} we define a binary relation $\cF\subset \cC_d \times \cC_d$,  For this we need two more notions.  First, we say that
$\xi,\xi'\in \cC_d$ are \emph{adjacent $d$-cubes} if $\xi\cap \xi'\in\cC_{d-1}$.  Second, we define the concept of {\em improper vertices}.   If a strand $v^\alpha=(v_1^\alpha,\ldots,v_{d+1}^\alpha)$ in $\vv$ obeys the condition $v_1^\alpha=v^\alpha_{d+1}$, then it corresponds to a vertex in $\cC$ via
\[
v^\alpha = [v_1^\alpha,v_1^\alpha]\times \ldots \times [v_{d}^\alpha,v_{d}^\alpha].
\]
We call these vertices in $\cC$ {\em improper} and write $G_\vv$ for the set of improper vertices.  Define the relation $\cF\subset \cC_d\times \cC_d$ as follows.  A pair $(\xi',\xi)\in \cF$ if
\begin{enumerate}
    \item $\xi',\xi$ are adjacent $d$-cubes and $\cross_\vv(\xi') \leq \cross_\vv (\xi)$, or
    \item $\xi,\xi'$ are adjacent $d$-cubes and $\xi,\xi'\in \st(v)$ for some $v\in G_\vv$.
\end{enumerate}

The relation $\cF$ can be considered as a directed graph, where there is an edge $\xi\to \xi'$ if $(\xi',\xi)\in \cF$, see Fig.~\ref{fig:lap}(b).  Recall that a directed graph is \emph{strongly connected} if every vertex is reachable from every other vertex.
A \emph{strongly connected component} of a directed graph  is a maximal subgraph that is strongly connected.
The set of strongly connected components forms a partition of the vertices of a directed graph
 $\cF$ and is denoted by
 $\sSC(\cF)$, cf.\ \cite{braids}. The set $\sSC(\cF)$ can be given a poset structure as follows.  First, define the \emph{condensation graph} of $\cF$ as the directed acyclic graph with vertices $\sSC(\cF)$; there is a directed edge between two strongly connected components if and only if there is an edge between  vertices in the respective strongly connected components. The  transitive, reflexive closure of the condensation graph provides the poset structure for 
$\bigl(\sSC(\cF),\le\bigr)$.  Given the relation $\cF$, any $d$-cube $\xi\in \cC_d$ belongs to some strongly connected component, denoted $[\xi]$. We can define a $\sSC(\cF)$-graded cell complex by defining $f_\vv\colon \cC\to \sSC(\cF)$ by 
\[
f_\vv(\xi) = \min_{\sSC(\cF)}\{ [\xi'] \mid \xi' \in \st(\xi)\cap \cC_d\} \in \sSC(\cF).
\]

From~\cite{braids} this is well-defined, order-preserving map.  The cell complex $\cC$ together with the map $f_\vv$ are the input to the algorithm $\CM$.  The output is a (graded chain homotopy equivalent) chain complex, called the \emph{Conley complex}, the boundary operator of which is called the \emph{connection matrix}. Briefly, the pair $(\cC,f_\vv)$ form a chain complex $(C_\bullet(\cC),\partial_\bullet)$, which is \emph{$\sSC(\cF)$-graded}, meaning (i) for each $n$, $C_n(\cC)$ admits a decomposition
\[
C_n(\cC) =  
\bigoplus_{p\in \sSC(\cF)} G_pC_n, 
\]
and (ii) the boundary operator $\partial_\bullet$ is $\sSC(\cF)$-filtered, i.e., it obeys the property
\[
\partial_n ( G_qC_n ) \subseteq \bigoplus_{p\leq q} G_pC_n,
\]
for all $n$ and $q$.  In the specific case of $(\cC,f_\vv)$ we have that
\[
C_n(\cC) = \bigoplus_{\xi\in \cC_n} \K\langle \xi\rangle  
= \bigoplus_{p\in \sSC(\cF)} G_pC_n, \quad\text{where}\quad G_p C_n
:= \bigoplus_{\xi\in f^{-1}_\vv(p)\cap \cC_n } \K\langle \xi\rangle.
\]

The Conley complex is an $\sSC(\cF)$-graded chain complex $(C_\bullet^\conley,\partial^\conley_\bullet)$, which (i) obeys the further condition that $\partial_n(G_q C_n)\subseteq \bigoplus_{p<q} G_pC_n$,
and (ii) is $\sSC(\cF)$-filtered chain equivalent to $(\cC_\bullet(\cC),\partial_\bullet)$, see~\cite{hms}.  The Conley complex is an invariant for $\sSC(\cF)$-graded chain complexes which is analogous to homology as an invariant for chain complexes, see~\cite{hms} for a detailed explanation. Akin to $\HOM$, $\CM$ uses multiple rounds of discrete Morse theory, and the templates (and the algorithm $\GMate$) apply only to the first round.   

\begin{rem}
See~\cite[Section 10]{hms} for a discussion of how the Conley complex can be regarded as an invariant upstream of persistent homology; cf.~\cite{mn} where graded Morse theory is used for preprocessing the computation of persistent homology.
\end{rem}

\subsection{Experiments for Braids}

The braid examples provide a set nice set of relevant examples for grading the cubical complex $\cC(m-1;d)$.  In particular, they enable the study of the performance of the templates with respect to scaling $\cC(m-1;d)$ with respect to both $m$ and $d$.

\begin{figure}[h!]
\begin{minipage}{.48\textwidth}
\centering
\begin{tikzpicture}[dot/.style={draw,circle,fill,inner sep=.75pt},line width=.5pt,scale=.5]
\foreach \x in {0,1,2,3,4,5,6,7,8}
    \foreach \y in {0,.75,1.5,2.5,3.25,4}
        \node (\x\y) at (\x, \y)[dot] {};
\foreach \y in {0,4}
    \foreach \x[count=\xi]  in {0,1,2,3,4,5,6,7}
        \draw (\x\y) to (\xi\y);
\foreach \x in {0,2,4,6} {
     \draw (\x, .75) to (\x+1, 2.5);
     \draw (\x, 1.5) to (\x+1, .75);
     \draw (\x, 2.5) to (\x+1, 3.25);
     \draw (\x, 3.25) to (\x+1, 1.5); }
\foreach \x in {1,3,5,7} {
     \draw (\x, .75) to (\x+1, 1.5);
     \draw (\x, 1.5) to (\x+1, 3.25);
     \draw (\x, 2.5) to (\x+1, .75);
     \draw (\x, 3.25) to (\x+1, 2.5); }
\end{tikzpicture}\\
(a)
\end{minipage}
\hspace{\fill}
\begin{minipage}{.48\textwidth}
\flushright
\centering
\begin{tikzpicture}[dot/.style={draw,circle,fill,inner sep=.75pt},line width=.5pt,scale=.5]
\foreach \x in {0,1,2,3,4}
    \foreach \y in {0,.75,1.75,2.5,3.25,4}
        \node (\x\y) at (\x, \y)[dot] {};
\foreach \x in {0,1,2,3,4}
    \node (\x) at (\x,1.25) {\tiny $\vdots$};
\foreach \y in {0,4}
    \foreach \x in {0,1,2,3}
        \draw (\x,\y) to (\x+1,\y);
\foreach \x in {0,1,2,3}
    \draw (\x,.75) to (\x+1,3.25);

\foreach \x in {0,1,2,3}
    \foreach \y in {2.5, 3.25}
        \draw (\x,\y) to (\x+1,\y-.75);
\end{tikzpicture}\\
(b)
\end{minipage}
\caption{(a) Braid diagram $\vv_4\in \Conf_8^6$, the 4-fold cover of $\vv$. (b) Schematic for $\bt^m\in \Conf_4^{m}$.}\label{fig:exp:braids}
\end{figure}

Given $\uu\in \Conf_d^m$ and the associated cubical complex $\cC=\cC(m-1;d)$ with grading $f_\uu\colon \cC\to \sSC(\cF)$, we record the following parameters, in Table \ref{tab:init:v} and \ref{tab:init:torus}, which we term {\em initial data}: 
\begin{enumerate}
    \item Either the dimension of the cubical complex $d$, or the parameter $m$,
    \item $\# \cC_d = (m-1)^d$, i.e., the number of top dimensional cubes ($d$-cubes),
    \item  $\# \cC = (2m-1)^d$, i.e., the total number of cells, 
    \item  $\# \sSC(\cF)$, i.e., the number of strongly connected components of $\cF$.
\end{enumerate}   

We record the results of two algorithms applied  to the graded complex $(\cC,f_\uu)$.  The first algorithm, $\GMorse$, uses $w=\GMate$ together with~\cite[Algorithm 3.12]{focm}, to compute the initial graded Morse complex $C_\bullet(w)$. This is the first round of discrete Morse theory within the algorithm $\CM$.  For $\GMorse$, we record the following results:
\begin{enumerate}
    \item $\# C_\bullet(w)=\sum_i \dim C_i(w)$, i.e., the size of the Morse complex produced by $\GMate$, 
    \item the execution time of the algorithm $\GMorse$.
\end{enumerate} 

Applying $\CM$ to $(\cC,f_\uu)$ produces a Conley complex $C^\conley_\bullet$ and we record the following results:
\begin{enumerate}
    \item $\# C^\conley_\bullet = \sum_i \dim C_i^\conley$, i.e., the size of the Conley complex,
    \item $\# \text{Tower}$, the number of rounds of discrete Morse theory needed to simplify to the Conley complex,
    \item the execution time of the algorithm $\CM$.
\end{enumerate}

In the next sections we describe the particular braids $\vv_n$ and $\bt^m$ which we study.  

\subsubsection{$n$-fold covers}
Let $\uu\in \Conf_d^m$ be a braid diagram. The \emph{$n$-fold cover of $\uu$}, denoted $\uu_n$, is obtained by $n$ copies of the braid $\uu$, with $\uu_n\in \Conf_{n\cdot d}^m$, as in Fig.~\ref{fig:exp:braids}(a).  We compute the Conley complexes for $\vv_n$ with $\vv$ of Figure~\ref{fig:braid}(a) for $1\leq n \leq 5$.  Studying the $n$-fold covers provides a nice test case of the template scheme with respect to dimension $d$ of the cubical complex $\cC(m-1;d)$.  

From Table \ref{tab:cm:v2} we see that $\GMate$ achieves a great reduction in the number of cells.   In particular $\#$Tower = 2, independent of $n$.  This means only one round of Morse theory is needed after computing the initial Morse complex.

\begin{table}[h!]
\setlength{\tabcolsep}{5pt}
\centering
\begin{tabular}{l l l l l} 
\toprule
  \multicolumn{5}{c @{}}{Initial Data for $n$-fold covers of $\vv$ }\\ 
\cmidrule(l){1-5}
 {$n$}   & {$d$ } & {\#$\cC_d$ } & {\#$\cC$  } & {\#$\sSC(\cF)$} \\
\midrule
1     & 2 & 25 & 121 & 13  \\
2 & 4 & 625 & 14641 & 114  \\
3  & 6 & 15625 & 1771561 & 879    \\
4  & 8 & 390625 & 214358881 & 7212  \\
5  & 10 & 9765625 & 25937424601 & 62157   \\
\bottomrule
\end{tabular}
\caption{Data for $\vv_n\in \Conf_{2n}^6$, the $n$-fold cover of $\vv$.}
\label{tab:init:v}
\end{table}

\begin{table}[h!]
\centering
\begin{subtable}[t]{0.4\textwidth}
\setlength{\tabcolsep}{5pt}
\centering
\begin{tabular}{l l l} 
\toprule
\multicolumn{3}{c @{}}{$\GMorse$ for $\vv_n$}\\ 
\cmidrule(l){1-3}
 {$n$}   & {\#$C_\bullet(w)$}  & {Time Elapsed} \\
\midrule
1   &   7       &  $65.8 \mu s \pm 60.6 ns$\\
2 & 133       &  $11.3ms \pm 36.2 \mu s$ \\
3 & 1825      & $1.53 s \pm 4.01 ms$  \\
4 &  23281    & $4m29s\pm 395ms$\\
5  & 291017    & $17h6m45s\pm 23m 36s$ \\
\bottomrule
\end{tabular}
\caption{}
\label{tab:mf:v}
\end{subtable}
\hspace{\fill}
\begin{subtable}[t]{.55\textwidth}
\setlength{\tabcolsep}{5pt}
\centering
\begin{tabular}{l l l l} 
\toprule
   \multicolumn{4}{c @{}}{$\CM$ for $\vv_n$ }\\ 
\cmidrule(l){1-4}
  {$n$}  & {\#$C_\bullet^\conley$ }  & {\#Tower } & {Time Elapsed}  \\
\midrule
1      & 3  & 2 & $76.3 \mu s \pm 1.57 \mu s$   \\
2  & 33  & 2  & $11.6ms\pm 13.3 \mu s$   \\
3 & 197  & 2 & $1.53s\pm 3.14 ms$     \\
4 & 1155 & 2 &  $4m26s \pm 657 ms$  \\
5 & 6727 & 2 & $16h45m47s\pm 4m9s$\\
\bottomrule
\end{tabular}
\caption{}
\label{tab:cm:v}
\end{subtable}
\caption{(a) Results for Morse complexes for $\vv_n$. (b) Results for Conley complexes for $\vv_n$.}
\label{tab:cm:v2}
\end{table}


\subsubsection{Torus Knot}




We also wish to examine how the templates scale on $\cC(m-1;d)$ with respect to $m$, i.e., the number of strands in the braid.  To do this, consider the braid diagram $\bt^m\in \Conf_4^m$, which is given schematically in Fig.~\ref{fig:exp:braids}(b). This braid diagram corresponds to a torus knot.  In this case, $d=4$ is fixed and $m$ is varied.   Regardless of the number of strands, the pattern of the torus knot remains the same: for any strand $t^\alpha,$ we have that the $t_{i+1}^\alpha$  is obtained from $t_i^\alpha$ via
\[
  t_{i+1}^\alpha =
  \begin{cases}
          t_{i}^\alpha-1 & 1 < \text{$t_{i}^\alpha < m-1$,}\\
          m-2 & \text{$t_{i}^\alpha = 1$}. 
  \end{cases}
\]

Finally, to agree with our assumptions on the braid diagrams, $\bt^m$ contains the constant strands $(0,\ldots,0)$ and $(m-1,\ldots,m-1)$.  The initial data for $\bt^m$ is in Table~\ref{tab:init:torus}. The results of the computations can be found in Tables~\ref{tab:c:torus}(a) and (b) where it can be seen that $\#\cC_\bullet(w)$ and $\#C_\bullet^\conley$ are independent of $n$.
A further check determined that the complexes themselves, $C_\bullet(w)$ and  $C_\bullet^\conley$, are actually independent of $n$.
This is a surprise; we know of no reason a priori to suspect this just from the form of the braid $\bt^m$.

\begin{table}[h!]
\setlength{\tabcolsep}{5pt}
\centering
\begin{tabular}{l l l l} 
\toprule
   \multicolumn{4}{c @{}}{Initial Data for $\bt^m$ }\\ 
\cmidrule(l){1-4}
    {$m$} & {\#$\cC_d$ } & {\#$\cC$  } & {\#$\sSC(\cF)$} \\
\midrule
10  &   6561  &  130321 &  624      \\
 20 &  130321 & 2313441 & 24154      \\
 40 & 2313441 & 38950081 &  587614   \\
 60 & 12117361 & 200533921  &  3386274   \\
 80 & 38950081 & 639128961 & 11396134\\
 100 & 96059601 & 1568239201 & 28873194\\
\bottomrule
\end{tabular}
\caption{Data for $\bt^m\in \Conf_4^m$.}
\label{tab:init:torus}
\end{table}



\begin{table}[h!]
\begin{subtable}[t]{0.4\textwidth}
\setlength{\tabcolsep}{5pt}
\centering
\begin{tabular}{l l l} 
\toprule
     \multicolumn{3}{c @{}}{$\GMorse$ on $\bt^m$}\\ 
\cmidrule(l){1-3}
 {$m$}   & {\#$C_\bullet(w)$}  & {Time Elapsed } \\
\midrule
10 & 81 & $58.2 ms \pm 344 \mu s$ \\
20   & 81  & $995 ms \pm 1.52 ms$ \\
40 &  81  & $17s \pm 145 ms $ \\
60 &  81   & $1m28s\pm 702 ms$\\
80 & 81 & $4m48s\pm 2.63 s$\\
100 & 81 & $12m9s\pm 6.34s$\\
\bottomrule
\end{tabular}
\caption{}
\end{subtable}
\hspace{\fill}
\begin{subtable}[t]{.55\textwidth}
\setlength{\tabcolsep}{5pt}
\centering
\begin{tabular}{l l l l} 
\toprule
    \multicolumn{4}{c @{}}{$\CM$ on $\bt^m$. }\\ 
\cmidrule(l){1-4}
    {$m$ }& {\#$C_\bullet^\conley$ }  & {\#Tower } & {Time Elapsed}  \\
\midrule
10 & 3  & 2 & $58.3ms \pm 88.2 \mu s$ \\
20      & 3   & 2 &  $1.01 s \pm 3.78 ms$ \\
40      & 3    &  2 &     $17s\pm 70.8 ms$   \\
60      & 3  & 2 &  $1m28s\pm 320ms$ \\
80 & 3  & 2 & $4m49s\pm 1.94s$ \\
100 & 3  & 2 & $12m9s\pm 7.83 s$ \\

\bottomrule
\end{tabular}
\caption{}
\end{subtable}
\caption{(a) Results for Morse complexes for $\bt^m$. (b) Results for Conley complexes for $\bt^m$.}\label{tab:c:torus}
\end{table}